\documentclass[final]{siamltex}
\usepackage{amssymb,latexsym,amsmath}
\usepackage{mathrsfs}
\usepackage{enumitem}
\usepackage{caption}
\usepackage{comment}
\usepackage{inputenc}
 \usepackage{makeidx}
\usepackage{color}

\allowdisplaybreaks

\newtheorem{remark}{Remark}[section]

\allowdisplaybreaks

\usepackage{epsfig} 
\usepackage{verbatim}

\hypersetup{colorlinks={true},linkcolor={blue},citecolor=green}

\usepackage{amssymb,nicefrac,bm,upgreek,mathtools,verbatim}
\usepackage{dsfont}
\usepackage{mathrsfs}

\RequirePackage[normalem]{ulem}

\usepackage{color}

\allowdisplaybreaks

\def\qed{\hfill $\diamond$}

\newcommand{\df}{\coloneqq}             
\DeclareMathOperator{\Exp}{\mathbb{E}}  
\DeclareMathOperator{\Pm}{\mathbb{P}} 
\newcommand{\sB}{{\mathscr{B}}}         
\newcommand{\RR}{\mathbb{R}}            
\newcommand{\Rd}{\mathbb{R}^d}          
\newcommand{\NN}{\mathbb{N}}            
\newcommand{\ZZ}{\mathbb{Z}}            

\newcommand{\sF}{{\mathfrak{F}}}   
\newcommand{\transp}{^{\mathsf{T}}}
\newcommand{\Act}{\mathbb{U}}           
\newcommand{\Uadm}{\mathfrak U} 
\newcommand{\Um}{\mathfrak U_{\mathsf{m}}}   
\newcommand{\Usm}{\mathfrak U_{\mathsf{sm}}}  

\newcommand{\Urc}{\mathfrak U_{\mathsf{RC}}} 
\newcommand{\Uwrc}{\mathfrak U_{\mathsf{WRC}}} 
\newcommand{\Uws}{\mathfrak U_{\mathsf{WS}}} 
\newcommand{\Udws}{\mathfrak U_{\mathsf{DWS}}} 
\newcommand{\cJ}{{\mathcal{J}}}  
\newcommand{\Prob}{{\mathcal{P}}}         
\newcommand{\grad}{\nabla}
\newcommand{\cB}{{\mathcal{B}}}  
\newcommand{\cC}{{\mathcal{C}}}   
\newcommand{\cD}{\mathcal{D}} 
 
\newcommand{\sL}{{\mathscr{L}}}  

\DeclareMathOperator*{\trace}{Tr}
\newcommand{\abs}[1]{\lvert#1\rvert}
\newcommand{\norm}[1]{\lVert#1\rVert}
\definecolor{dmagenta}{rgb}{.4,.1,.5}
\definecolor{dblue}{rgb}{.0,.0,.5}
\definecolor{mblue}{rgb}{.0,.0,.7}
\definecolor{ddblue}{rgb}{.0,.0,.4}
\definecolor{dred}{rgb}{.7,.0,.0}
\definecolor{dgreen}{rgb}{.0,.5,.0}
\definecolor{Eeom}{rgb}{.0,.0,.5}

\allowdisplaybreaks

\begin{document}

\title{Controlled Diffusions under Full, Partial and Decentralized Information: Existence of Optimal Policies and Discrete-Time Approximations}



\author{Somnath Pradhan\footnotemark[1]
\and Serdar Y\"{u}ksel\footnotemark[2] \ \footnotemark[3]}
\footnotetext[1]{Department of Mathematics,
Indian Institute of Science Education and Research Bhopal, Bhopal, MP - 462066, India (\email{somnath@iiserb.ac.in})}
\footnotetext[2]{Department of Mathematics and Statistics, Queen's University, Kingston K7L 3N6, ON, Canada, (\email{yuksel@queensu.ca})}
\footnotetext[3]{The research of S.Y. was partially supported by the Natural Sciences and Engineering Research Council of Canada (NSERC)}

%



\maketitle

\begin{abstract}
We present existence and discrete-time approximation results on optimal control policies for continuous-time stochastic control problems under a variety of information structures. These include fully observed models, partially observed models and multi-agent models with decentralized information structures. While there exist comprehensive existence and approximations results for the fully observed setup in the literature, few prior research exists on discrete-time approximation results for partially observed models. For decentralized models, even existence results have not received much attention except for specialized models and approximation has been an open problem. Our existence and approximations results lead to the applicability of well-established partially observed Markov decision processes and the relatively more mature theory of discrete-time decentralized stochastic control to be applicable for computing near optimal solutions for continuous-time stochastic control.
\end{abstract}



\section{Introduction}

In this paper, we will present existence and discrete-time approximation results on optimal control policies for continuous-time stochastic control under a variety of information structures through a unified perspective. These will include fully observed models, partially observed models and multi-agent models with decentralized information structures. 

In the continuous-time theory, the analysis in the literature has primarily been restricted to fully observed and, to some extent, partially observed models: For fully observed models, there is an extensive literature, and we refer the reader to \cite{kushner1990numerical,kushner2001numerical, kushner2012weak} \cite{BJ-06, JPR-19P, KN99A, KN98A, KN2000A} for a detailed analysis and review. We will relax some of the technical regularity conditions studied in these papers. For partially observed models \cite{FlPa82} has arrived at existence results via {\it belief}-separation, and \cite{bismut1982partially} has arrived at existence results, through an approach which avoids separation. These papers adopted a {\it wide-sense admissible control policy} framework (see \cite[Definition 1.3]{bismut1982partially} and \cite[p. 264 and Lemma 2.3]{FlPa82}) which essentially places a {\it Young topology} on the space of measurement and control action processes with a fixed, Brownian, marginal on the measurement process and with a martingale condition on the measurements, given the past control and measurements, ensuring that the control policies are non-anticipative. We note that the topology adopted in \cite{kushner1990numerical,kushner2001numerical,kushner2012weak}, where one considers a Young topology (instead of the joint product space of measurements and actions) on the time and action spaces with the random measures defined on this space having a fixed marginal on the time variable, and the measures being independent of future driving noise invites a discrete-time approximation more naturally (compared with the formulation \cite{bismut1982partially}\cite{FlPa82}) since time appears as a separate parameter. For uncontrolled partially observable continuous-time models, the authors in  \cite{MasiRung81,di1982approximation, KHJFil79}, established discrete-time approximation results\,. In particular, \cite{MasiRung81,di1982approximation, KHJFil79} prove that the discrete-time model weakly converges to the continuous-time model as the parameter of the discretization approaches to zero\,. By using an updating/prediction procedure, a similar discretization approach is proposed in \cite{MRTSAP09} and used to study an allocation problems in finance in the context of a non-stationary economy. Again, for the uncontrolled model, in \cite{Crisan2022} the authors studied a high order approximations of the solution of the stochastic filtering problem and derived their path-wise representation\,. Also, they established the robustness property of the discretized solutions\,.  


While existence results for fully observed and partially observed models are well established, for decentralized information structures few research results have been reported. In particular, we will establish existence of optimal policies and near optimality of discrete-time approximations with a unified perspective. Utilizing Girsanov's change of measure technique as presented in \cite{benevs1971existence} and an accompanying continuity result \cite[p. 455]{benevs1971existence} that incorporates more stringent control topology and convexity requirements, a pertinent area of research within the continuous-time decentralized framework is represented by the studies \cite{charalambous2016decentralized,charalambous2014equivalence,charalambous2014maximum}. The change of measure arguments, which were adopted by Witsenhausen in discrete-time decentralized stochastic control \cite{wit88} (see reviews and further generalizations in \cite{hogeboom2021sequential},\cite[Chapter 3]{YukselBasarBook},\cite{YukselWitsenStandardArXiv}), was studied in continuous-time stochastic control in these studies. For continuous-time stochastic differential decision system with decentralized noisy information structure, in \cite{CharalambousA13}, the authors established existence of a team optimal policy by applying Girsanov's change of measure argument; where, as an intermediate step the Radon-Nikodym derivative with respect to a reference measure is shown to be continuous as a function of policy under the weak* topology (see \cite[Lemma~1]{CharalambousA13})\, where the measurement functions $h^i$ are uniformly Lipschitz over $\Rd$ and a global growth condition is required, with both the diffusion matrix $\upsigma$ and the the measurement functions $h^i$ allowed to be functions of control policies as well. Different from these studies, we consider complementary system models, present a different control topology through which we establish the key result on the continuity and approximation of the Radon-Nikodym derivative in control policies (which was proved in \cite[Lemma~1]{charalambous2016decentralized}, using $L^2$ convergence of the policies, under the assumption that the information structures are not corrupted by external noise processes). Additionally, owing to the flexibility of the topology we consider, we establish near optimality of discrete-time models (which was not studied in the aforementioned studies). We also consider models with global state dynamics and local measurements, as well as local states with dynamics coupling. Further recent studies on continuos-time stochastic team problems include \cite{huang2022general,jackson2023approximately}, with the former being on the LQG setup and the latter studying near optimality of decentralized control policies in the limit of large numbers of agents.

Thus, we present an alternative, probabilistic, unified approach on both existence as well as approximation results for such problems in a large generality of information structures and system models. 

Our analysis allows one to utilize the relatively more mature theory of decentralized stochastic control in discrete-time: There is an established theory on existence, structure, approximation, and learning theory for optimal discrete-time decentralized stochastic control:

Witsenhausen \cite{wit88} and Ho and Chu \cite{HoChu}  have presented complementary methods for reducing a dynamic decentralized stochastic control problem involving multiple agents to one which has static information, that is, one in which no agent's information is affected by the policies of any other agent. Building on such static reduction, \cite[Theorem 5.2]{YukselWitsenStandardArXiv}, \cite[p. 1691]{gupta2014existence}, \cite[Theorem 2.9]{YukselSaldiSICON17}, \cite[Theorem 5.6]{YukselWitsenStandardArXiv} have obtained general existence results in such problems. A key aspect is with regard to the weak sequential compactness properties of strategic measures facilitated by static reduction \cite{saldiyukselGeoInfoStructure}, which leads to both existence results as well as near optimality results for solutions obtained by finite model approximations \cite{saldiyuksellinder2017finiteTeam}. Thus, discrete-time approximations for a continuous-time problem facilitate numerical and learning theoretic results. The same comment applies for partially observed models, as there exist many analytical, numerical and effective learning theoretic methods for discrete-time POMDPs (see e.g. \cite{kara2021convergence} and the references therein for a comprehensive review).


{\bf Contributions and Main Results.}
\begin{itemize}
\item[(i)] We present a unified treatment of existence and discrete-time approximation results on optimal control policies for continuous-time stochastic control problems under a variety of information structures. For fully observed models, we first review the {\it weak convergence} approach presented in the literature by various authors, and notably Kushner and colleagues, though with some relaxations (such as on the absence of Lipschitz continuity conditions on the drift term) and present an existence and approximation result in Theorem \ref{existencePiece-wise}. We then present an alternative approach where measure transformation is adopted in Theorem \ref{existencewithMeasTransformation}. This leads to a method for both existence and approximations to be utilized for stochastic control under more general information structures as noted further below, and with complementary conditions (such as the cost being continuous only in the action variable). 
%
\item[(ii)] Our primary contributions are with regard to partial and decentralized information. For partially observed models, while existence results had been established earlier, we establish near optimality of piece-wise constant control policies using the unified approach of the paper in Theorem \ref{existenceApprPartialInf}. 
\item[(iii)] For the case with decentralized information structures, we present new results for both existence and approximations under two general models in Theorems \ref{existenceApprDecenLocalMeas} and \ref{existenceApprCoupledDynLocalState}. Following this, we present discrete-time approximation results, which is new to the literature, to our knowledge.
\item[(iv)] As a second main set of results, we show that for stochastic control problems with full, partial or decentralized information, one can obtain discrete-time controlled Markov model whose solutions are near optimal for the continuous-time system. While corresponding results (under stronger conditions) for the fully observable models are available in the literature, the results for partial information and multi-agent setups are new in the context of their generality, and invite the applicability of relatively more mature discrete-time theory (in the context of analytical, approximation, and learning theoretic results) to continuous-time models.
\end{itemize}

\subsection*{Notation}
\begin{itemize}
\item $\sB_{r}$ denotes the open ball of radius $r$ in $\Rd$, centered at the origin.
\item By $\trace S$ we denote the trace of a square matrix $S$.
\item $m_{[0,t]} := \{m_s : 0\leq s \leq t\}$.
\item $\textbf{X} = (X^1, X^2,\dots , X^N)$, $\textbf{m} = (m_1, m_2,\dots , m_N)$. 
\item For any domain $\cD\subset\Rd$, the space $\cC^{k}(\cD)$ ($\cC^{\infty}(\cD)$), $k\ge 0$, denotes the class of all real-valued functions on $\cD$ whose partial derivatives up to and including order $k$ (of any order) exist and are continuous.
\item $\cC_{\mathrm{c}}^k(\cD)$ denotes the subset of $\cC^{k}(\cD)$, $0\le k\le \infty$, consisting of functions that have compact support. This denotes the space of test functions.
\item $\cC_{b}(\Rd)$ denotes the class of bounded continuous functions on $\Rd$\,.
\item $\cC^{k}_{0}(\cD)$, denotes the subspace of $\cC^{k}(\cD)$, $0\le k < \infty$, consisting of functions that vanish in $\cD^c$.
\end{itemize}
\section{Fully Observed Model}\label{sectionFullyObserved}
We start the discussion with the fully observed model. We note again that such a setup has already been studied extensively (see e.g. \cite{kushner2001numerical}, \cite{kushner1990numerical}). For fully observed models, while for the case without measure transformation we essentially follow the machinery present by Kushner and colleagues, we also consider and present an approach where measure transformation is adopted and a continuity result is presented. This leads to an alternative approach for both existence and approximations. A key continuity lemma, Lemma \ref{SomnathLemma} is presented toward this end. Our analysis will set the stage for the more general information structures to be studied later in the paper. Consider a continuous-time process $X_t$ taking values in a Euclidean space $\mathbb{R}^{d}$, controlled by a control process $U_t$ taking values in a compact metric space $\mathbb{U}$. In the context of a diffusion process, $X_t$ is a solution of the following stochastic differential equation (SDE) defined on a complete probability space $(\Omega, \sF, \Pm)$
\begin{align}\label{diffProcess}
dX_t = b(X_t,U_t)dt + \upsigma(X_t)dW_t
\end{align}
driven by standard Brownian motion $W_t$, under the control policy $U$ and the initial condition $x\in \Rd$. We also allow the control policy to be randomized, that is ${\cal P}(\mathbb{U})$-valued, where  ${\cal P}(\mathbb{U})$ denotes the space of probability measures on $\mathbb{U}$ under the topology of weak convergence. An \emph{admissible control} is a ${\cal P}(\mathbb{U})$-valued process $U_t$ which satisfies the following non-anticipativity condition: for $s<t\,,$ $W_t - W_s$ is independent of
$$\sF_s := \,\,\mbox{the completion of}\,\,\, \sigma(X_0, U_r, W_r : r\leq s)\,\,\,\mbox{relative to} \,\, (\sF, \mathbb{P})\,.$$ Let $\Uadm$ denote the space of all admissible controls. An admissible control is said to be Markov if $U_t = v(t, X_t)$ for some measurable map $v: \RR_+\times \RR^d \to {\cal P}(\mathbb{U})$\,. Let $\Um$ denote the space of all Markov controls\,.

To ensure existence and uniqueness of weak solutions of \cref{diffProcess}, we impose the following assumptions on the drift $b$ and the diffusion matrix $\upsigma$\,.
\begin{itemize}
\item[\hypertarget{A1}{{(A1)}}]
\emph{Continuity \& Boundedness condition:\/} The function $b$ is jointly continuous in $(x, \zeta)$ and $\upsigma$ is locally Lipschitz continuous, i.e., for some constant $C_{R}>0$ depending on $R>0$, we have
\begin{equation}\label{ELocalLipschA}
\norm{\upsigma(x_1) - \upsigma(x_2)}^2 \,\le\, C_{R}\,\abs{x_1 - x_2}^2
\end{equation}
for all $x_1,x_2\in \sB_R$ and $\zeta\in\Act$, where $\norm{\upsigma}\df\sqrt{\trace(\upsigma\upsigma\transp)}$\,. Also, we assume that $b$, $\upsigma$ are uniformly bounded, i.e., for some positive constant $\hat{C}_0$, we have
\begin{equation*}
\sup_{\zeta\in\Act}\, \abs{b(x, \zeta)} + \norm{\upsigma(x)} \,\leq\, \hat{C}_0 \qquad \forall\, x\in\RR^{d}\,.
\end{equation*}

\medskip
\item[\hypertarget{A2}{{(A2)}}]
\emph{Nondegeneracy:\/} For some $\hat{C}_1 >0$, it holds that
\begin{equation*}
\sum_{i,j=1}^{d} a^{ij}(x)z_{i}z_{j}
\,\ge\,\hat{C}_{1} \abs{z}^{2} \qquad\forall\, x\in \RR^{d}\,,
\end{equation*}
and for all $z=(z_{1},\dotsc,z_{d})\transp\in\RR^{d}$,
where $a\df \frac{1}{2}\upsigma \upsigma\transp$\,.
\end{itemize} In view of \hyperlink{A2}{{(A2)}}, it is easy to see that $\upsigma^{-1}$ exists and it is bounded\,. Using the Girsanov's change of measure technique as in \cite[Theorem~2.2.11]{ABG-book} we have that under the hypotheses \hyperlink{A1}{{(A1)}}--\hyperlink{A2}{{(A2)}}, for each $U\in \Uadm$ there exists a unique weak solution of \cref{diffProcess} (see also \cite[Section~4, pp. 190-197]{IkedaWatanabe89}, \cite{Krylov})\,.

We note that Assumptions \hyperlink{A1}{{(A1)}}--\hyperlink{A2}{{(A2)}} can be weakened in many situations. By following Girsanov's change of measure technique as in \cite[Theorem~2.2.11]{ABG-book}, under the assumption that the drift term $b$ is measurable and satisfies the linear growth condition, and $\upsigma$ is non-degenerate, locally Lipschitz continuous with linear growth condition, one can prove that (\ref{diffProcess}) admits unique weak solution for any $U\in \Uadm$\,.

In the following, first we will trace some of the ideas presented by Kushner (see e.g. \cite{kushner2012weak}), though with some presentational differences and then present an alternative approach.

Suppose that one wishes to minimize the cost
\begin{align}\label{cost1}
\cJ_{T}(x, U) = \Exp_x^{U}\bigg[\int_0^T c(X_s, U_s) ds  + c_T(X_T)\bigg],
\end{align}
over all admissible control policies, where $\Exp_x^{U}$ denotes the expectation with respect to the law of the process \cref{diffProcess} under the control policy $U$ with initial condition $x$\,. Here, the running cost function $c:\Rd\times \Act \to \RR$, and the terminal cost $c_T:\Rd \to \RR$ are assumed to be bounded and continuous functions.

We define a relaxed wide-sense admissible control policy in the following. We first place the Young topology which is a topology on stochastic kernels defining randomized/relaxed policies, where one views policies to be probability measures on a product space with a fixed marginal on input space (for more details, see\cite[p.254]{warga2014} and \cite[Definition 3.2]{yuksel2023borkar}). Let ${\cal P}_{\lambda}([0,T]\times\mathbb{U})$ be the space of positive finite measures on $[0,T] \times \mathbb{U}$ with its fixed marginal on $[0,T]$ to be the Lebesgue measure ${\lambda}$. Let
\begin{align}
\Urc :=& \{m: m \,\,\text{is a progressively measurable random process taking values in}\nonumber\\  &{\cal P}_{\lambda}([0,T]\times\mathbb{U}) \,\, \text{such that} \, m_{[0,s]}\,\,\text{is independent of}\, W_t - W_s\nonumber\\ 
&\text{for any} \, 0\leq s < t \leq T \}
\end{align} Thus $\Urc$ denotes the space of all relaxed wide-sense admissible control policy. Following the discussion as in \cite[p. 52]{KushnerSingular1990}, it follows that any $m\in \Urc$ can be represented as $m(dt,d\zeta) = m_t(d\zeta)\lambda(dt)$, where $m_t(d\zeta)$ is the random process such that $m_t(B)$ is adapted to $\sF_t$ for any $B\in \mathcal{B}(\Act)$\,. 

For above control problem (see, eq. \cref{cost1}), policy $m^*\in \Urc$ is said to be optimal if 
\begin{align*}
\cJ_{T}(x, m^*) = \inf_{m\in \Urc} \cJ_{T}(x, m)\,.
\end{align*}

We then consider the $C([0,T];\mathbb{R}^{d})$-valued (under sup-norm) $X_t$ process (solution to the diffusion equation (\ref{diffProcess})) induced by $m(dt,d\zeta)(\omega)$ and then consider the space of probability measures on these random variables. From \cite[Theorem~2.2.11]{ABG-book}, it is easy to see that under any choice of control process $m(dt,d\zeta)(\omega)$, (\ref{diffProcess})) admits a unique weak solution\,.

In order to establish existence of an optimal policy and suitable approximation of it, we adopt the following two approaches.

\subsection{Two Approaches: With or without Measure Transformation}

\subsubsection{Without measure transformation}\label{withoutMeasureTransformSection}

In one approach, presented extensively by several seminal studies by Kushner and collaborators \cite{kushner2012weak,kushner2001numerical,kushner1990numerical} as well as others such as Borkar \cite{borkar1989topology}, one considers the following.
Given the above, let us introduce a parametric family of elliptic operator, which will be useful in our analysis. With $\zeta\in\Act$ treated as a parameter, we define a family of operators ${\sL}_{\zeta}$ mapping $\cC^2(\Rd)$ to $\cC(\Rd)$ by
\begin{equation*}
{\sL}_{\zeta} f(x) \df\, \trace\bigl(a(x)\grad^2 f(x)\bigr) + \,b(x,\zeta)\cdot \grad f(x) \,.  
\end{equation*}

In the following theorem, we follow \cite{kushner2012weak,kushner2001numerical,kushner1990numerical} but under relatively weaker conditions, e.g., $b$ is jointly continuous and $\upsigma$ is locally Lipschitz continuous.

\begin{theorem}\label{contConvPathM}
Suppose that Assumptions \hyperlink{A1}{{(A1)}} and \hyperlink{A2}{{(A2)}} hold. If $m^n \to m$ weakly as $n\to \infty$, with $X^n$ is the solution of \cref{diffProcess} under the policy $m^n$, i.e., it satisfies
\begin{align*}\label{TContidiffPro1}
dX_t^n = b(X_t^n,m_t^n)dt + \upsigma(X_t^n)dW_t
\end{align*}
Then the solution process $X^n$ converges weakly to $X$ which is a unique solution of  
\begin{align*}\label{TContidiffPro2}
dX_t = b(X_t, m_t)dt + \upsigma(X_t)dW_t\,.
\end{align*}
\end{theorem}
\textbf{Proof.} We will follow Kushner's {\it weak convergence} approach (see e.g. \cite{kushner2001numerical}), with a slightly different presentation. 
Let $m^n \to m$ weakly as $n\to \infty$.  

Let $X^n$ be the unique weak solution of (\ref{diffProcess}) (see \cite[Theorem~2.2.11]{ABG-book}) corresponding to the control $m^n$. Suppose $p,q$ and $t_i\leq s$ for $i\leq p$ are arbitrary and $\Phi\in \cC_{c}\left((\Rd)^p\times (\RR)^{pq}\right)$, $\psi_j \in \cC\left([0, T]\times \Act\right)$ for $j\leq q$\,. Then by standard martingale characterization of the solution of the SDE (\ref{diffProcess}), by It\^{o} formula and standard monotone class argument as in \cite[p. 41]{ABG-book} (see also, \cite[Proposition~2.1, p. 169]{IkedaWatanabe89}), we get
\begin{align}\label{EcontConvPathM1A}
\Exp_{x}^{m^n} \bigg[ \Phi &\left(X_{t_i}^n, \int_{0}^{t_i}\int_{\Act}\psi_j(s, \zeta)m_s^n(d\zeta)\lambda({ds}), i\leq p, j\leq q\right)\bigg(f(X^n_t) - f(X^n_s)\nonumber\\
-& \int_s^t {\sL}_{\zeta} f(X^n_s) m^n_s(d \zeta) \lambda(ds)\bigg)\bigg]=0\,,
\end{align}
for each twice continuously differentiable function $f$ with compact support.

Now, for any $\delta \geq 0$ and any $0\leq t \leq s \leq t+\delta$, it follows that 
\begin{align*}
\Exp_{x}^{m^n}\big[|X_{s}^{n} - X_{t}^{n}|^2\big] \leq &  \Exp_{x}^{m^n}\big[|\int_{t}^{s} \left(\int_{\Act}b(X_r^n, \zeta)m_r^n(d\zeta) \lambda(d r) + \upsigma(X_r^n)dW_r\right)|^2\big]\nonumber \\
\leq & 2(\|b\|_{\infty}^2 \delta +  \|\upsigma\|_{\infty}^2)\delta\,.
\end{align*} Therefore, by a standard tightness criterion (see \cite[Theorem~7.2, p. 81]{PBill-book}, \cite[Theorem 1.4.2]{kushner2012weak}), we have set of laws of $X_{\cdot}^{n}$ is tight in $\mathcal{P}(C([0,T];\mathbb{R}^{d}))$\,. Let $X$ be the limit of a weakly convergent subsequence of $X^n$ (this exists by Prohorov's theorem \cite[Theorem~5.1, p. 59]{PBill-book}). By considering a subsequence where both $X^n, m^n$ converge weakly, by Skorohod's representation theorem (see \cite[Theorem~6.7]{PBill-book}), we have over a common probability space $X^n(\omega) \to X(\omega)$ as a $C([0,T];\mathbb{R}^{d})$-valued process and $m^n(\omega)(\cdot, \cdot) \to m(\omega)(\cdot, \cdot)$ weakly as a ${\cal P}_{\lambda}([0,T]\times\mathbb{U})$-valued process, almost surely (in $\omega$). Accordingly, by an application of the generalized dominated convergence \cite[Theorem 3.5]{serfozo1982convergence}, we have the integral
\begin{equation*}
\bigg( \int_{\Act}\,\left(b(X^n_s(\omega), \zeta)\cdot \grad f(X_s^n(\omega)) \right)m^n_s(\omega)(d \zeta) + \trace\bigl(a(X^n_s(\omega))\grad^2 f(X^n_s(\omega))\bigr)\bigg)
\end{equation*}
converge almost surely to
\begin{equation*}
\bigg( \int_{\Act}\,\left(b(X_s(\omega), \zeta)\cdot \grad f(X_s(\omega))\right)m_s(\omega)(d \zeta) + \trace\bigl(a(X_s(\omega))\grad^2 f(X_s(\omega))\bigr)\bigg)
\end{equation*}
Thus, taking limit $n\to \infty$ (along this subsequence), from (\ref{EcontConvPathM1A}) we deduce that the limiting process $X$ satisfies the following
\begin{align*}
\Exp_x^{m} \bigg[ \Phi &\left(X_{t_i}, \int_{0}^{t_i}\int_{\Act}\psi_j(s, \zeta)m_s(d\zeta)\lambda({ds}), i\leq p, j\leq q\right)\bigg(f(X_t) - f(X_s)\nonumber\\
-& \int_s^t {\sL}_{\zeta} f(X_s) m_s(d \zeta) \lambda(ds)\bigg]=0.
\end{align*}
Now, by the martingale characterization as in \cite[Theorem~2.2.9]{ABG-book}, we have $X$ is a unique weak solution of (\ref{diffProcess}) under the limiting policy $m$\,. We thus conclude that the path measure is weakly continuous in the control policy. 
\qed

\begin{remark}
In view of \cite[Theorem~2.2.2]{ABG-book}, it is easy to see that we can relax the uniform boundedness assumption in \hyperlink{A1}{{(A1)}}, which is used to prove the tightness of the sequence of solutions, to local Lipschitz continuity condition with linear growth (see \cite[Corollary~2.3.9]{ABG-book}).    
\end{remark}

Thus, the map from $m$ to the associated solution of the SDE \ref{diffProcess} is continuous under the topology defined on $\Urc$. The problem is to find an optimal $m\in \Urc$ which minimizes (\ref{cost1}). From this, one can, as in the deterministic case summarized in \cite[Section~7.1]{saldiyukselGeoInfoStructure}, arrive at general conditions for the existence of an optimal solution. We can rewrite the cost as
\begin{align}
\cJ_T(x, m) = \Exp_{x}^{m}\bigg[\int_0^T  \int_{\Act} c(X_s, \zeta)m_s(d\zeta)\lambda({ds}) + c_T(X_T)\bigg]
\end{align}
with $\cJ_T(x, \cdot): \Urc \to \mathbb{R}$\,. Next theorem proves that the total cost as a function of policy is continuous\,.
\begin{theorem}\label{contCostPathM}
Suppose that Assumptions \hyperlink{A1}{{(A1)}} and \hyperlink{A2}{{(A2)}} hold. Then, the map $$m\mapsto \cJ_T(x, m)$$ from $\Urc$ to $\mathbb{R}$ is continuous. 
\end{theorem}

\textbf{Proof.} 
The continuity of $\cJ_T$ in $m$ can be established again via the Skorohod representation argument presented in the proof of Theorem \ref{contConvPathM}. Let $m^n \to m$ weakly as $n\to \infty$. Once again, by a tightness argument, and by Theorem \ref{contConvPathM}, we have that as $m^n \to m$, then the associated solutions $X^n \to X$ (weakly) which is the solution of the SDE \cref{diffProcess} corresponding to the policy $m$. By considering a subsequence where both converge weakly, by Skorohod's representation theorem (see \cite[Theorem~6.7]{PBill-book}), we deduce that over a common  probability space $m^n(\omega) \to m(\omega)$ weakly as a ${\cal P}_{\lambda}([0,T]\times\mathbb{U})$-valued process and $X^n(\omega) \to X(\omega)$ as a $\cC([0, T]; \Rd)$-valued process, almost surely (in $\omega$). Then, by the generalized dominated convergence \cite[Theorem 3.5]{serfozo1982convergence}, we deduce that the expected cost
$$\Exp_{x}^{m^n}\bigg[ \int_0^T  \int_{\mathbb{U}} c(X^n_s,\zeta)m^n_s(d\zeta)\lambda(ds) + c_T(X_T^n)  \bigg]$$
converges to
$$\Exp_{x}^{m}\bigg[ \int_0^T  \int_{\mathbb{U}} c(X_s,\zeta)m_s(d\zeta)\lambda(ds) + c_T(X_T) \bigg].$$ \qed

Next, using the above continuity result we want to prove the near optimality of piece-wise constant policies\,. These policies break the time horizon into smaller intervals, during which the policy remains constant. This means that within each interval, the action taken or the decision made does not change, but it can switch to a different action or decision at the boundaries of these intervals. Piece-wise constant policies are useful because they simplify the decision-making process by reducing the number of times the policy needs to be updated. They play important role in scenarios where frequent policy changes are impractical or costly.

\begin{theorem}\label{existencePiece-wise}
Suppose that Assumptions \hyperlink{A1}{{(A1)}} and \hyperlink{A2}{{(A2)}} hold.  Then
\begin{itemize}
\item[(i)] There exists an optimal control policy in $\Urc$\,.
\item[(ii)] For every $\epsilon > 0$, there exists an ordinary piece-wise constant control policy in $\Urc$ (thus also non-anticipative) which is $\epsilon$-optimal. 
\end{itemize}
\end{theorem}
\begin{proof}
Note that the space $\Urc$ is compact, thus existence follows from Theorem \ref{contCostPathM}. 

From \cite[Theorem 2.3.1]{ABG-book}, we know that the set of ordinary non-anticipative measures with quantized support (where both time and control variables are discretized into finite sets of values) is dense in ${\cal P}_{\lambda}([0,T]\times\mathbb{U})$. Thus, by the continuity of the cost as a function of policies (as we have established in Theorem~\ref{contCostPathM}), we obtain our result\,.  
\end{proof}
\begin{remark}
Since the map $m$ to the corresponding solution $X$ of the stochastic differential equation (SDE) is continuous (see Theorem~\ref{contConvPathM}). To demonstrate the existence and approximation result as outlined in Theorem~\ref{existencePiece-wise}, it suffices to assume that both the running cost $c$ and the terminal cost $c_T$ are bounded and lower semicontinuous. Given this assumption of lower semicontinuity, we can apply Fatou's lemma to show that the map $m \to \cJ_T(x, m)$ is lower semicontinuous. Consequently, the result is derived from the compactness of the policy space $\Urc$. \end{remark}

\subsubsection{With measure transformation}\label{withMeasureTransformSection}

In an alternative approach, we consider the state process to be exogenous and the control only impacting the cost function. For the analysis of this subsection, we assume the following regularity conditions for the cost functions:  
\begin{itemize}
\item we assume that the running cost $c(x, \zeta)$ is bounded measurable and continuous in its second argument (i.e., only in $\zeta$) and $c_T$ is bounded measurable.
\end{itemize}

To study this, we adopt Girsanov's measure transformation, building on \cite{benevs1971existence}. This enables us to transform the probability measure $\Pm$ to an equivalent probability measure $\Pm_0$, under which the SDE becomes
\[dX'_t = \upsigma(X'_t)dW_t,\]
which in fact is policy independent. 

Somewhat differently from the previous subsection, we define relaxed wide-sense admissible control policy in the following. As in the above, we first place the Young topology on the control action space, by viewing the joint distribution of $X^{'}$ and the control policy $m$ to be a probability measure on $C([0,T]; \mathbb{R}^{d}) \times {\cal P}_{\lambda}([0,T]\times\mathbb{U})$ with its fixed marginal on $C([0,T]; \mathbb{R}^{d})$ to be the Wiener measure, moreover we require that, under the measure-transformed model, $m_{[0,s]}$ be independent of $W_t - W_s$ for any $t > s$\,. Let $\Uwrc$ denotes the space of all wide-sense admissible control policies\,. A typical element of $\Uwrc$ is denoted by $m$ (without loss of generality)\,. 

For any $t\geq 0$, let
\begin{align*}
Z_t = \exp \biggl[ \int_0^t \upsigma^{-1}(X_s)b(X_s,m_s)dW_s - \frac{1}{2} \int_0^t |\upsigma^{-1}(X_s)b(X_s,m_s)|^2 ds \biggr]\,. 
\end{align*} Since $b$ is bounded and $\upsigma^{-1}(x)$ exists and is bounded (which is a consequence of \hyperlink{A2}{{(A2)}}), by Novikov's criterion (see \cite[Theorem~5.3  p. 152]{IkedaWatanabe89}), we have $\{Z_t\}$ is a nonnegative $\sF_t$ martingale with mean one and therefore a legal family of Radon-Nikodym derivatives\,. By Girsanov's theorem, we define the probability measure $\Pm_0$ as follows: if $\Pm_t$ (resp. $\Pm_{0t}$) denotes the restriction of $\Pm$ (resp. $\Pm_{0}$ ) to $\sF_t$ then $\Pm_t \ll\Pm_{0t}$ with (for details see \cite{benevs1971existence})
\begin{align*}
\frac{d\Pm_t}{d\Pm_{0t}} := Z_t\,. 
\end{align*}

In this case, the marginal on the state process is fixed, but the cost function is now represented as:
\[ \cJ_T(x, m) = \Exp_{x,\Pm_0}^{m}\bigg[\frac{d\Pm_T}{d\Pm_{0T}} \bigg(\int_0^T c(X_s, m_s) ds + c_T(X_T) \bigg) \bigg],\] where $\Exp_{x,\Pm_0}^{m}$ is the expectation with respect to the probability measure $\Pm_0$\,. 

In view of the above discussions it follows that in the transformed model (i.e., under $\Pm_0$), the solution of the SDE does not depend on the choice of control policy\,. Thus, one wishes to minimize
\begin{align}\label{OPTR11}
\inf_{m \in \Uwrc} \Exp_{x,\Pm_0}^{m}\bigg[ \frac{d\Pm_{T}}{d\Pm_{0T}} \bigg( \int_0^T  \int_{\mathbb{U}}  c(X_s,\zeta)m_s(d\zeta)\lambda(ds)  + c_T(X_T)\bigg)\bigg]\,.
\end{align}
\begin{remark}[Comparison between the approaches: With measure transformation vs. without measure transformation]
Note that in the measure transformation case the measure on the path space of the solution of the SDE is control-free and thus fixed, so the set of measures in ${\cal P}(C([0,T];\RR^d))$ characterizing law of the solution of the SDE is a singleton. There is no need for establishing tightness towards showing Theorem \ref{contConvPathM}, unlike the approach without measure transformation presented above. In the latter approach, this is not needed and the analysis is purely probabilistic, even though a new lemma on the continuity of the reduced cost under the transformed strategic measure, Lemma \ref{SomnathLemma}, is needed. This approach brings about versatility for studying more general information structures, as we will see in the later sections. However, for the measure transformation, $\upsigma(\cdot)$ is not allowed to depend on control. As we show later in this paper, the Girsanov's change of measure technique allows us to tackle stochastic control problems with a variety of information structures, e.g., fully observable, partially observable, and decentralized information. It is important to note that the former approach (without measure transformation) could be extended to the degenerate case, control dependent $\upsigma$ and infinite dimensional problems, etc. Note that in the former approach, one has to show that the set on the path space changes continuously with the control convergence (that is, law of the solution depends on $m$). This requires a martingale characterization of the solution process and a tedious stochastic analysis.  
\end{remark}

In the following, we adopt the latter approach, as it will be much simpler to be generalized to information structures beyond the fully observed model, including decentralized information structures. This analysis is based on the supporting result in Lemma \ref{SomnathLemma}, which shows that the Radon-Nikodym derivative is continuous as a function of policies (under a suitable topology over the policy space)\,. Similar continuity result proved in \cite[Lemma~1]{charalambous2016decentralized}, using $L^2$ convergence of the policies, under the assumption that the information structures are not corrupted by any external noise processes\,. In \cite{CharalambousA13}, the authors established existence of a team optimal policy by applying Girsanov's change of measure argument\,. In order to do so, as an intermediate step they prove that the Radon-Nikodym derivative is continuous as a function of policy under the weak* topology (see \cite[Lemma~1]{CharalambousA13})\,.

\begin{lemma}\label{compactnessLemma11}
The space $\Uwrc$ under the weak convergence topology is compact.
\end{lemma}
\textbf{Proof.}  Note that for $m\in\Uwrc$ the joint distribution of $(X^{'}, m)$ is an element of ${\cal P} \left( C([0,T]; \mathbb{R}^{d}) \times {\cal P}_{\lambda}([0,T]\times\mathbb{U})\right)$. Since under any $m\in \Uwrc$ the marginal of the joint distribution of $(X^{'}, m)$ on $C([0,T]; \mathbb{R}^{d})$ is fixed and ${\cal P}_{\lambda}([0,T]\times\mathbb{U})$ is compact (via Prohorov's theorem), it follows that the law of the elements of $\Uwrc$ is tight. Thus relatively compact by Prohorov's theorem. Let $m^n\in \Uwrc$ be such that $m^n\to m$. Since the marginal of the joint distribution of $(X^{n'}, m^n)$ on $C([0,T]; \mathbb{R}^{d})$ is Wiener measure for all $n\in \NN$, it follows that the the marginal of the joint distribution of the limit $(X^{'}, m)$ on $C([0,T]; \mathbb{R}^{d})$ is also Wiener measure\,. As in the the proof of \cite[Lemma 2.3]{FlPa82}, since independence is preserved under weak convergence of probability measures (see also \cite[Theorem 5.6]{YukselWitsenStandardArXiv}), thus $\Uwrc$ is also closed, hence compact\,.
\qed

Note that for any function $f:\Act \to \Rd$ and $m\in \mathcal{P}(\Act)$, the function $f(m) = \int_{\Act} f(\zeta) d\zeta$\,. The following lemma shows that the Radon-Nikodym derivative is continuous as a function of policies over $\Uwrc$ (under the topology of weak convergence)\,.
\begin{lemma}\label{SomnathLemma}
Suppose that Assumptions \hyperlink{A1}{{(A1)}} and \hyperlink{A2}{{(A2)}} hold. Then, on $\Uwrc$, the map
\begin{eqnarray*}
m\mapsto \exp \biggl[ \int_0^T \upsigma^{-1}(X_s)b(X_s,m_s)dW_s - \frac{1}{2} \int_0^T |\upsigma^{-1}(X_s)b(X_s,m_s)|^2 ds \biggr]
\end{eqnarray*}
is continuous in $L^1$ norm.
\end{lemma}

\textbf{Proof.} See the Appendix.\qed

Since the running $c$ and the terminal cost $c_T$ are bounded and continuous, using this continuity property of the Radon-Nikodym derivative as a function of policy, we arrive at the following continuity result.

\begin{lemma}\label{contLemma11}
Suppose that Assumptions \hyperlink{A1}{{(A1)}} and \hyperlink{A2}{{(A2)}} hold. Then, the map
$$m\mapsto \cJ_T(x,m)$$ is continuous on $m \in \Uwrc$ under the weak convergence topology.
\end{lemma}

\textbf{Proof.} We apply Lemma \ref{SomnathLemma}. As $m^n \to m$ weakly, we have a probability space in which $m^n(\omega) \to m(\omega)$ almost surely. Let $$Z_T^n = \exp \biggl[ \int_0^T \upsigma^{-1}(X_s)b(X_s,m_s^n)dW_s - \frac{1}{2} \int_0^T |\upsigma^{-1}(X_s)b(X_s,m_s^n)|^2 ds\bigg]$$ and $$Z_T = \exp \biggl[ \int_0^T \upsigma^{-1}(X_s)b(X_s,m_s)dW_s - \frac{1}{2} \int_0^T |\upsigma^{-1}(X_s)b(X_s,m_s)|^2 ds\bigg]\,.$$ Now by the triangle inequality, we have 
\begin{align*}
&|\Exp \bigg[ Z_T^n \big[\int_0^T\int_{\Act} c(X_s, \zeta)m^n(d s, d \zeta) + c_T(X_T)\big]\bigg]\nonumber\\
&\,\, -  \Exp \bigg[Z_T \big[\int_0^T\int_{\Act} c(X_s, \zeta)m(d s, d \zeta) + c_T(X_T)\big]\bigg]|\nonumber\\
\leq & \Exp \bigg[\bigg | Z_T^n \int_0^T\int_{\Act} c(X_s, \zeta)m^n(d s, d \zeta) -  Z_T \int_0^T\int_{\Act} c(X_s, \zeta)m^n(d s, d \zeta)\bigg| \bigg]\nonumber\\
& \,\,+ \Exp \bigg[\bigg| Z_T \int_0^T\int_{\Act} c(X_s, \zeta)m^n(d s, d \zeta) -  Z_T \int_0^T\int_{\Act} c(X_s, \zeta)m(d s, d \zeta)\bigg|\bigg] \nonumber\\
&\,\, + \Exp\bigg[\big|Z_T^nc_T(X_T)-  Z_Tc_T(X_T)\big|\bigg]\,.
\end{align*}
Since $c, c_T$ are bounded, by Lemma~\ref{SomnathLemma}, we have that the first and third terms of the above inequality go to zero. Furthermore, since $Z_T$ is bounded and $c(x, \cdot)$ is bounded and continuous, it is easy to see that the second term of the above inequality also goes to zero as $n\to \infty$\,. This completes the proof\,.  \qed

The following theorem proves the existence of an optimal policy in $\Uwrc$. Also, it shows that the piece-wise constant policies in $\Uwrc$ are near optimal\,.

\begin{theorem}\label{existencewithMeasTransformation}
Suppose that Assumptions \hyperlink{A1}{{(A1)}} and \hyperlink{A2}{{(A2)}} hold. Then, we have
\begin{itemize}
\item[(i)] There exists an optimal control policy in $\Uwrc$.
\item[(ii)] For every $\epsilon > 0$, there exists an ordinary piece-wise constant control policy in $\Uwrc$ (thus also non-anticipative) which is $\epsilon$-optimal. 
\end{itemize}
\end{theorem}

\textbf{Proof.} (i) Existence follows from Lemma \ref{contLemma11} and Lemma \ref{compactnessLemma11}: The value is continuous in this joint measure on $\{(X_s, m_s), s \in [0,T]\}$ and this set of measures is tight, where the control policy space $\Uwrc$ is compact under weak convergence topology (see e.g. the proof of Lemma~\ref{compactnessLemma11}). These lead to the compactness-continuity conditions and accordingly an existence result for optimal policies follows. 

(ii)
Applying chattering lemma \cite[Theorem 2.3.1]{ABG-book} \footnote{See \cite[Section~6]{FlPa82} for a similar construction of piece-wise constant control policy. See also \cite[Theorem 3]{milgrom1985distributional}, \cite{balder1997consequences},\cite[Proposition 2.2]{beiglbock2018denseness},
\cite{lacker2018probabilistic}, \cite{borkar1988probabilistic}, \cite[Chapter 7]{castaing2004young} or \cite[Theorem 5.1]{yuksel2023borkar} for related arguments though without the non-anticipativity condition on the approximating policies with discrete-support; the problem of finding such an approximate sequence is crucial for discrete-time optimization, which will be studied further in the paper.}, we deduce that ordinary policies are dense in relaxed ones. In-particular we have that the set of non-anticipative measures ${\cal P}_{\lambda}([0,T]\times\mathbb{U})$ (with fixed marginal on $C([0, T]\times \mathbb{U})$) which have quantized support (in both time and control) is dense in $\Uwrc$ (for more details see \cite[Section~4]{FlPa82}). Now, considering the joint law of the sate and control, we have the continuity of the expected total cost as a function of control policy (Lemma~\ref{contLemma11}). Thus by the continuity and denseness (of the ordinary piece-wise constant control policies), we get the near optimality of the piece-wise constant control policies. \qed

Later in the paper, the result above will allow us to approximate a continuous-time process with a (sampled) discrete-time process and the machinery developed for discrete-time optimal control will be applicable.

\section{Partially Observed Setup}\label{partiallyObsSec}

Consider now a partially observed continuous-time process $\{X_t\}$ on $\mathbb{R}^{d}$, controlled by a control process $\{U_t\}$ taking values in a compact action space $\mathbb{U}$, and with an associated observation process $\{Y_t\}$ taking values in $\mathbb{R}^M$, where $0 \leq t \leq T$. The evolution of $\{X_t,Y_t\}$ is given by the stochastic differential equations
\begin{align}\label{EPartialStat1}
dX_t &= b(X_t,U_t) dt + \upsigma(X_t) dW_t, \nonumber\\
dY_t &= g(X_t) dt + dB_t\,,
\end{align}
where, $g:\RR^{d}\to \RR^{M}$ is a continuous and bounded function and $W$ and $B$ are independent standard Wiener processes with values in $\mathbb{R}^d$ and $\mathbb{R}^M$, respectively (hence, $\upsigma$ is a  $d \times d$-matrix). The objective is to minimize the following cost function
\begin{align}\label{criterionPOMDPCT}
\Exp_{x}^{U} \bigg[ \int_0^T c(X_t,U_t) dt + c_T(X_T) \bigg], 
\end{align}
where the running cost $c: \mathbb{R}^d \times \mathbb{U} \rightarrow [0,\infty)$ and the terminal cost $c_T:\mathbb{R}^d \rightarrow [0,\infty)$ are given functions. Throughout this section we are assuming that the running cost $c$ and the terminal cost $c_T$ are bounded and continuous functions.

We require that the control process $\{U_t\}$ be adapted to the filtration generated by the observation process $\{Y_t\}$; that is, for each $t \in [0,T]$, $U_t$ is $\sigma\left(Y_s, 0 \leq s \leq t \right)$-measurable. We will call such policies (\emph{strict-sense or precise}) admissible policies. In \cite{FlPa82}, Fleming and Pardoux introduced another class of policies which they named to be \textit{wide-sense admissible policies}. Using this relaxed class of policies, they studied the existence of optimal policies to the above problem. 

The idea is again to first apply Girsanov's transformation so that the measurements $Y_t$ form an independent Wiener process under new probability measure $Q$\,. Following Fleming and Pardoux \cite[p. 264]{FlPa82}, we define an admissible control as a probability measure on $C([0,T]\times \mathbb{R}^{M})\times {\cal P}_{\lambda} ([0,T] \times \mathbb{U})$ with its fixed marginal on $C([0,T]\times \mathbb{R}^{M})$ be the Wiener measure. In addition, under the new measure $Q$, $Y_r - Y_t$ is independent of $\{X_0, W_{\cdot}, Y_s, m_s; s\leq t\}$, for any $0\leq t \leq r \leq T$\,. The latter condition states that actions up to time $t$ are independent of the observations after time $t$ given past observations and actions. In other words, instead of saying that actions should be dependent on current and past observations, this condition states that actions should be independent of future observations given past observations and actions.  Let $\Uws$ denote the space of such policies, where we endow this space with the weak convergence topology. As in Lemma~\ref{compactnessLemma11}, we have $\Uws$ is compact under weak convergence topology\,. Without loss of generality, a typical element of $\Uws$ is denoted by $m$.

Let 
\[G_t:= e^{\int_0^t g(X_s)dY_s - \frac{1}{2}\int_0^t |g(X_s)|^2 ds}\]
with $\int_0^t g(X_s)dY_s$ being a stochastic integral with respect to the random process $Y_s$. Since $g$ is bounded, by Novikov's criterion (see \cite[Theorem~5.3  p. 152]{IkedaWatanabe89}), we have $\{G_t\}$ is a nonnegative martingale with mean one. Therefore it is legal family of Radon-Nikodym derivatives\,. Now by Girsanov's theorem, we define the probability measure $Q$ as follows: (for details see \cite{FlPa82})
\begin{align*}
\frac{d\Pm}{dQ} := G_T\,, 
\end{align*} under this new measure $Q$ we get an equivalent model
\begin{align}\label{EPartialStat2}
dX_t &= b(X_t,U_t) dt + \upsigma(X_t) dW_t, \nonumber\\
dY_t &= dB_t\,.
\end{align}

This relation allows us to view the partially observed problem as one with independent measurements, with the dependence pushed to the Radon-Nikodym derivative. 

\begin{theorem}\label{existenceApprPartialInf}
Suppose that Assumptions \hyperlink{A1}{{(A1)}} and \hyperlink{A2}{{(A2)}} hold\,. Then,
\begin{itemize}
\item[(i)] The total cost
\begin{align}
\cJ_T(x, m) = \Exp_{x,Q}^{m} \bigg[ \frac{d{\Pm}}{d{Q}}\left(\int_0^T c(X_t,m_t) dt + c_T(X_T)\right) \bigg], 
\end{align}
as a function of policies is continuous over the space of wide-sense admissible policies $\Uws$.
\item[(ii)] There exists an optimal control policy in $\Uws$.
\item[(iii)] For every $\epsilon > 0$, there exists an ordinary piece-wise constant control policy in $\Uws$ which is $\epsilon$-optimal. 
\end{itemize}
\end{theorem}

\textbf{Proof.} (i)-(ii) Once we have a measure transformation, continuity and existence follow by the similar argument as in Lemma \ref{contLemma11}, Lemma \ref{compactnessLemma11} and Theorem~\ref{contConvPathM} (since the signal process $X_{\cdot}$ is not fixed in the current setup): For any sequence $m^n \to m$ in $\Uws$, by the similar argument as in the proof of Theorem~\ref{contConvPathM}, we have that the solution $X^n$ of (\ref{EPartialStat2}) associated to $m^n$ converges (weakly) to the solution $X$ (of (\ref{EPartialStat2})) associated to $m$\,. Also, since $g$ is bounded and continuous, arguing as in the proof of Lemma~\ref{SomnathLemma}, it follows that $\frac{d{P}}{d{Q}}(X_{[0,t]}^n,Y_{[0,t]})$ converges to $\frac{d{P}}{d{Q}}(X_{[0,t]},Y_{[0,t]})$\,. Thus, following the steps as in the proof of Lemma~\ref{contLemma11}, we obtain the continuity of the value\,. Also, the set of admissible control measures is weakly closed as the independence of the variables $Y,U$ is preserved under weak convergence (see the proof of \cite[Lemma 2.3]{FlPa82}). These lead to the compactness-continuity conditions and accordingly an existence result for optimal policies follows. 

(iii) 
As in \cite[Section~4, p. 274]{FlPa82}, we have the set of non-anticipative measures ${\cal P}_{\lambda}([0,T]\times\mathbb{U})$ with fixed marginal on $C([0, T]; \mathbb{R}^{M})$ which have quantized support, in both time and control, is dense (see also Theorem \ref{existencewithMeasTransformation})\,. Thus, in view of the above continuity result it follows that piece-wise constant control policies which are furthermore non-anticipative are near optimal.\qed

Accordingly, we can again establish both existence and discrete approximation results. Once again, the above will allow us to approximate a continuous-time process with a (sampled) discrete-time process and the machinery developed for discrete-time optimal control will be applicable.

\section{Decentralized Information}\label{DecSecContOpt}

\subsection{Decentralized Model with Local Measurements}\label{CoupledLocalMeas}

Consider a continuous-time process $\{X_t\}$ on a Euclidean space $\mathbb{R}^{d}$, controlled by a collection of control process ${\bf U}_t:= (U^k_t, i = 1, \cdots, N)$ with each $U^i_t$ taking values in a compact Borel action space $\mathbb{U}^i$, and with an associated observation process $\{Y^i_t\}$ taking values in $\mathbb{R}^M$, where $0 \leq t \leq T$. Let ${\bf Y}=(Y^1,\cdots,Y^N)$. The evolution of $\{X_t,Y^i_t, i=1,\cdots,N\}$ is given by the stochastic differential equations
\begin{align}
dX_t &= b(X_t,U^1_t,\cdots,U^N_t) dt + \upsigma(X_t) dW_t, \nonumber \\
dY^i_t &= g^i(X_t) dt + dB^i_t, \qquad i=1,\cdots, N.\label{sdeDec1}
\end{align}
Where, $W$ and $B^i, i=1,\cdots, N$ are independent standard Wiener processes with values in $\mathbb{R}^d$ and $\mathbb{R}^M$, respectively and $g^{i}:\RR^{d}\to \RR^{M}$, $i=1,\dots ,  N$ are given continuous and bounded functions\,.

In this section, we assume that the drift term $b$ and the diffusion matrix $\upsigma$ satisfies similar conditions as in \hyperlink{A1}{{(A1)}} and \hyperlink{A2}{{(A2)}}. In particular, they satisfy the following:
\begin{itemize}
\item[\hypertarget{D1}{{(D1)}}]
\emph{Local Lipschitz continuity:\/}
The function $b\colon\RR^d\times \prod_{k=1}^{N}\Act^{k}\to\RR^d$ is jointly continuous in $(x, {\bf \zeta})$ and $\upsigma\,=\,\bigl[\upsigma^{ij}\bigr]\colon\RR^{d}\to\RR^{d\times d}$ is locally Lipschitz continuous in $x$. In particular, for some constant $C_{R}>0$
depending on $R>0$, we have
\begin{equation*}
\norm{\upsigma(x_1) - \upsigma(x_2)}^2 \,\le\, C_{R}\,\abs{x_1 - x_2}^2
\end{equation*}
for all $x_1,x_2\in \sB_R$, where $\norm{\upsigma}\df\sqrt{\trace(\upsigma\upsigma\transp)}$\,.

\medskip
\item[\hypertarget{D2}{{(D2)}}]
\emph{Boundedness:\/}
The functions $b$ and $\upsigma$ are uniformly bounded, i.e., 
\begin{equation*}
\sup_{{\bf \zeta}\in \prod_{k=1}^{N}\Act^{k}}\, \abs{b(x, {\bf\zeta})} + \norm{\upsigma(x)}^{2} \,\le\, C \qquad \forall\, x\in\RR^{d}\,,
\end{equation*}
for some constant $C>0$.

\medskip
\item[\hypertarget{D3}{{(D3)}}]
\emph{Nondegeneracy:\/}
For some $\hat{C}_1 > 0$, it holds that
\begin{equation*}
\sum_{i,j=1}^{d} a^{ij}(x)z_{i}z_{j}
\,\ge\,\hat{C}_1 \abs{z}^{2} \qquad\forall\, x\in \RR^{d}\,,
\end{equation*}
and for all $z=(z_{1},\dotsc,z_{d})\transp\in\RR^{d}$, where $a\df \frac{1}{2}\upsigma \upsigma\transp$\,.
\end{itemize}

The objective here is to minimize the following cost function
\begin{align}\label{decCostCrit}
\Exp_{x}^{{\bf U}} \bigg[ \int_0^T c(X_t,U^1_t,\cdots,U^N_t) dt + c_T(X_T) \bigg], 
\end{align}
where the running cost $c: \mathbb{R}^d \times \prod_{k=1}^{N}\mathbb{U}^{k} \rightarrow [0,\infty)$ and the terminal $c_T:\mathbb{R}^d \rightarrow [0,\infty)$ are given functions. Throughout this section we assume that the cost functions $c, c_T$ are bounded continuous. 

As in the partially observable case, let 
\[G_t:= \prod_{i=1}^N e^{\int_0^t g^i(X_s)dY^i_s - \frac{1}{2}\int_0^t |g^i(X_s)|^2 ds}\]
with $\int_0^t g^i(X_s)dY^i_s$ being a stochastic integral with respect to the random process $Y_s$. Since $g^i, i= 1,\dots , N$ are bounded, by Novikov's criterion (see \cite[Theorem~5.3  p. 152]{IkedaWatanabe89}), we have $\{G_t\}$ is a nonnegative martingale with mean one. Thus it is legal family of Radon-Nikodym derivatives\,. Now by Girsanov's theorem, we define the probability measure $Q$ as follows:
\begin{align*}
\frac{d\Pm}{dQ} := G_T\,,
\end{align*} under this new measure $Q$ we get an equivalent decoupled measurement model: 
\begin{align}
dX_t &= b(X_t,U^1_t,\cdots,U^N) dt + \upsigma(X_t) dW_t, \nonumber \\
dY^i_t &= dB^i_t, \qquad i=1,\cdots, N.\label{sde11}
\end{align}

This relation allows us to view the decentralized stochastic control problem as one with independent measurements, with the dependence pushed to the Radon-Nikodym derivative.

We define the admissible policies to be the following: we define an admissible control as a probability measure on $C([0,T]\times \mathbb{R}^{M})\times {\cal P}_{\lambda} ([0,T] \times \mathbb{U}^i)$ with its fixed marginal on $C([0,T]\times \mathbb{R}^{M})$ be the Wiener measure. In addition, under the new measure $Q$, $Y_r^i - Y_t^i$ is independent of $\{X_0, W_{\cdot}, Y_s^i, m_s^i; s\leq t\}$, for any $0\leq t \leq r \leq T$ and independent from all $(m_{\cdot}^k,Y_{\cdot}^k), k \neq i$\,. Let $\Udws$ denote the space of such {\it decentralized wide sense admissible} policies, where we endow this space with the weak convergence topology\,. Without loss of generality, a typical element of $\Udws$ is denoted by ${\bf m} = (m^1, \cdots, m^N)$\,.

We state the following on existence and near optimality of discrete-time approximations.  

\begin{theorem}\label{existenceApprDecenLocalMeas}
Suppose that Assumptions \hyperlink{D1}{{(D1)}}, \hyperlink{D2}{{(D2)}} and \hyperlink{D3}{{(D3)}} hold. Then, we have 
\begin{itemize}
\item[(i)] The total cost
\begin{align}\label{newCostTiltedMeasure}
\cJ_T(x,{\bf m}) = \Exp_{x, {Q}}^{{\bf m}} \bigg[&\prod_{i=1}^N (e^{\int_0^t g^i(X_s)dY^i_s - \frac{1}{2}\int_0^t |g^i(X_s)|^2 ds)} \bigg(\int_0^T c(X_t,m^1_t,\cdots,m_t^N) dt\nonumber\\
& + c_T(X_T) \bigg) \bigg], 
\end{align}
as a function of policy is continuous over the space of wide-sense admissible policies $\Udws$.
\item[(ii)] There exists an optimal control policy in $\Udws$.
\item[(iii)] For every $\epsilon > 0$, there exists an ordinary piece-wise constant control policy in $\Udws$ which is $\epsilon$-optimal. 
\end{itemize}
\end{theorem}
\textbf{Proof.} (i)-(ii) Once we have a decoupled model, continuity and existence follows by the similar argument as in the proof of Lemma \ref{contLemma11}, Lemma \ref{compactnessLemma11} and Theorem~\ref{contConvPathM} (since the signal process $X_{\cdot}$ is not fixed in this setup): Let $\{{\bf m}^n\}$ be a sequence in $\Udws$ such that ${\bf m}^n \to {\bf m}$. Then arguing as in the proof of Theorem~\ref{contConvPathM}, it is easy to see that the solution $X^n$ of (\ref{sde11}) associated to ${\bf m}^n$ converges to the solution $X$ of (\ref{sde11}) associated to ${\bf m}$\,. Also, since $g^i$, $i=1,\dots ,N$ are bounded and continuous, arguing as in the proof of Lemma~\ref{SomnathLemma}, it follows that $\frac{d{P}}{d{Q}}(X_{[0,T]}^n,\bf{Y}_{[0,T]})$ converges to $\frac{d{P}}{d{Q}}(X_{[0,T]},\bf{Y}_{[0,T]})$\,. Now, following the steps as in the proof of Lemma~\ref{contLemma11}, we obtain the continuity of the value\,. 

Each ${\bf m}\in \Udws$ is a product measure across agents, and this set of measures is weakly compact as each coordinate is so (e.g. by \cite[Theorem 4.5]{saldiyukselGeoInfoStructure}) and since the independence of the measures on $Y^i_t,U^i_t$ from future increments is preserved under weak convergence, see e.g. the proof of \cite[Lemma 2.3]{FlPa82}). 
These lead to the compactness-continuity conditions and accordingly an existence result for optimal policies follows. 

(iii) For each agent, by showing that the set of non-anticipative measures ${\cal P}_{\lambda}([0,T]\times\mathbb{U})$ (with fixed marginal on $C([0, T]\times \mathbb{R}^{M})$ be the Wiener measure) which have quantized support (in both time and control) is dense (as in Theorem \ref{existencewithMeasTransformation}), we have that piece-wise constant control policies which are non-anticipative are near optimal for the product measure as well. This allows one to approximate a continuous-time process with a (sampled) discrete-time process and the machinery developed for discrete-time optimal control will be applicable. 
\qed

\subsection{Decentralized Model with Coupled Dynamics and Local State}\label{CoupledDyLocalSt}
Instead of (\ref{sdeDec1}), consider a collection of $N$ agents with coupled dynamics given as
\begin{align}
dX^i_t &= b^i(X^i_t,U^i_t)dt + b^i_0({\bf X}_t, {\bf U}_t) dt + \upsigma^i(X_t^i) dW^i_t, \qquad i=1,\cdots, N; \label{sdeDecLoc1} 
\end{align}
Here, $b^i : \RR^d\times \Act^{i}\to \RR^{d}$, $b^i_{0}: (\RR^{d})^N \times \prod_{k=1}^{N}\Act^{k} \to \RR^d$ and $\upsigma^{i}: \RR^{d} \to \RR^{d\times d}$ are given functions and $W^i$ are independent standard Wiener processes with values in $\mathbb{R}^d$\,, $i=1,\cdots, N$\,.
We assume that the functions $b^i, b^i_0, \upsigma^i$ for $i=1,\cdots, N$, satisfy the following:
\begin{itemize}
\item[\hypertarget{D1}{{$\hat{(D1)}$}}]
\emph{Local Lipschitz continuity:\/}
For $i=1,\cdots, N$, we have $b^i : \RR^d\times \Act^{i}\to \RR^{d}$ and $b^i_{0}: (\RR^{d})^N \times \prod_{k=1}^{N}\Act^{k} \to \RR^d$ are jointly continuous and for some constant $C_{R}>0$ depending on $R>0$, and for all $x_1,x_2\in \sB_R$, we have
\begin{equation*}
\norm{\upsigma^i(x_1) - \upsigma^i(x_2)}^2 \,\le\, C_{R}\,\abs{x_1 - x_2}^2
\end{equation*}

\item[\hypertarget{D2}{{$\hat{(D2)}$}}]
\emph{Boundedness:\/}
The functions $b^i$,  $b_0^i$ and $\upsigma^i$, $i=1,\cdots, N$, are uniformly bounded, i.e.,for some constant $C>0$,
\begin{equation*}
\sup_{\zeta\in \Act^{i}}\, \abs{b^i(x, \zeta)} + \norm{\upsigma^i(x)}^{2} \,\le\, C \qquad \forall\, x\in\RR^{d}\,,
\end{equation*}
\begin{equation*}
\sup_{{\bf \zeta}\in \prod_{k=1}^{N}\Act^{k}}\, \abs{b_0^i({\bf x}, {\bf \zeta})}\,\le\, C \qquad \forall\, {\bf x}\in(\RR^{d})^N\,,
\end{equation*}

\item[\hypertarget{D3}{{$\hat{(D3)}$}}]
\emph{Nondegeneracy:\/}
For some $\hat{C}_1 > 0$, it holds that
\begin{equation*}
\sum_{i,j=1}^{d} a^{k, ij}(x)z_{i}z_{j}
\,\ge\,\hat{C}_1 \abs{z}^{2} \qquad\forall\, x\in \RR^{d}\,,\, k=1,\cdots, N\,, 
\end{equation*}
and for all $z=(z_{1},\dotsc,z_{d})\transp\in\RR^{d}$, where $a^k\df \frac{1}{2}\upsigma^{k} (\upsigma^k)^{\transp}$\,.
\end{itemize} In view of \hyperlink{D3}{{$\hat{(D3)}$}}, it is easy to see that $(\upsigma^i)^{-1}$ exists and is bounded for all $i=1,\cdots, N$\,.

The objective here is to minimize the following cost function
\begin{align}\label{decCostCritA}
\Exp_{x}^{{\bf U}} \bigg[ \int_0^T c({\bf X}_t,U^1_t,\cdots,U_t^N) dt + c_T({\bf X}_T) \bigg], 
\end{align}
where the running cost $c: (\mathbb{R}^d)^N \times \prod_{k=1}^{N}\mathbb{U}^{k} \rightarrow [0,\infty)$ and the terminal cost $c_T:(\mathbb{R}^d)^N \rightarrow [0,\infty)$ are assumed to be bounded continuous functions\,. Also, we assume that the control policies are only locally measurable, that is $U^i_t$ is measurable with respect to $\sigma(X^i_{[0,t]})$ for all $t \in [0,T]$.

Since $\upsigma^i$ is invertible (follows from \hyperlink{D3}{{$\hat{(D3)}$}}), for $i=1,\cdots, N$, we can rewrite (\ref{sdeDecLoc1}) as
\begin{align}
dX^i_t &= \upsigma^i(X_t^i) (\upsigma^i)^{-1}(X_t^i)\left(b^i(X^i_t,U^i_t) +  b^i_0({\bf X}_t, {\bf U}_t) \right)  dt + \upsigma^i(X_t^i) dW^i_t\,. \label{sdeDecLoc11} 
\end{align} Since $b^i, b_0^i$ and $(\upsigma^i)^{-1}$ are bounded, by Novikov's criterion (see \cite[Theorem~5.3  p. 152]{IkedaWatanabe89}), we have
$$\hat{Z}_t := \prod_{k=1}^{N}e^{( \int_0^T  (\upsigma^i)^{-1}(X_t) \left(b^1(X^i_s,U^i_s) +  b^i_0({\bf X}_s, {\bf U}_s) \right) dW_t^i-\frac{1}{2} \int_0^T |(\upsigma^i)^{-1}(X_t) \left(b^1(X^i_s,U^i_s) +  b^i_0({\bf X}_s, {\bf U}_s) \right)|^2 dt)}.$$ is a nonnegative martingale with mean one. Therefore it is legal family of Radon-Nikodym derivatives\,. Now by Girsanov's theorem as in \cite{SelkYuksel2023}, we define a new probability measure $Q$ as follows:
$$\frac{d\Pm}{d Q}=\hat{Z}_T$$ 
and under this new measure $Q$ we get the following decoupled (non-interacting) agent model: 
\begin{align}
dX^i_t &=  \upsigma^i(X_t^i) dW^i_t, \qquad i=1,\cdots, N\,. \label{sdeDecLocNI1} 
\end{align}

The idea is now to push the dependence in the dynamics to a dependence in the cost. Such an application arises in mean-field models. In this case, by Lemma \ref{SomnathLemma} we obtained on the continuous dependence of the Radon-Nikodym derivative on control policies will be useful. As in Section \ref{withMeasureTransformSection}, we define a relaxed wide-sense admissible control policy by first placing the Young topology on the control action space, by viewing the control process to be a probability measure on $C([0,T]; \mathbb{R}^d)\times {\cal P}_{\lambda}([0,T] \times \mathbb{U}^i)$ with its fixed marginal on $C([0,T]; \mathbb{R}^d)$ to be the Wiener measure\,. We require also that $m^i_{[0,t]}$  be independent of $W^i_s - W^i_t, s > t$ for every $t \in [0,T]$ and independent from $m^j,W^j_s, \,\, j \neq i, s \in [0,T]$. We call again such policies decentralized locally wide sense admissible policies, and denote with $\Udws$. Without loss of generality, a typical element of $\Udws$ is denoted by ${\bf m} = (m^1, \cdots, m^N)\,$. Also, it is easy to see that in the trasformed model, the measure on the path space is fixed\,.

Now, following the analysis in Theorem \ref{existencewithMeasTransformation}, \ref{existenceApprDecenLocalMeas}, we have the following theorem, which proves the existence of a team optimal policy and near optimality of piece-wise constant control policy\,.

\begin{theorem}\label{existenceApprCoupledDynLocalState}
Suppose that Assumptions \hyperlink{D1}{{$\hat{(D1)}$}}, \hyperlink{D2}{{$\hat{(D2)}$}} and \hyperlink{D3}{{$\hat{(D3)}$}} hold. Then,
\begin{itemize} 
\item[(i)] The total cost
\begin{align}
\cJ_T({x, \bf m}) =& \Exp_x^{{\bf m}} \bigg[ \hat{Z}_T\bigg(\int_0^T c({\bf X}_t,m^1_t,\cdots,m^N) dt + c_T({\bf X}_T) \bigg) \bigg], \label{newCostTiltedMeasureDec}
\end{align}
is continuous over the space of wide-sense admissible policies $\Udws$.
\item[(ii)] There exists an optimal control policy in $\Udws$.
\item[(iii)] For every $\epsilon > 0$, there exists an ordinary piece-wise constant control policy in $\Udws$ which is $\epsilon$-optimal. 
\end{itemize}
\end{theorem}

\textbf{Proof.} (i)-(ii) Once we have a decoupled model, continuity and existence follows by similar argument as in Lemma \ref{contLemma11}, Lemma \ref{compactnessLemma11} and Theorem~\ref{existencewithMeasTransformation}. As in the proof of Theorem \ref{existenceApprDecenLocalMeas}, the weak convergence of marginal measures implies the weak convergence of the product measures (e.g. by \cite[Theorem 4.5]{saldiyukselGeoInfoStructure}). Accordingly, the value is continuous in ${\bf m}$. Note that the continuity of the exponential Radon-Nikodym derivative term in the policies follows as discussed in Lemma \ref{SomnathLemma}. Thus, by the continuity of cost as a function of control policy and the compactness of $\Udws$ ensures the existance of optimal control policy in $\Udws$\,.

(iii) For each agent, by showing that the set of measures ${\cal P}_{\lambda}([0,T]\times\mathbb{U})$ (with fixed marginal on $C([0, T]; \mathbb{R}^d)$ to be Wiener measure) which have quantized support (in both time and control) is dense (see Theorem \ref{existencewithMeasTransformation}), we have that piece-wise constant control policies, which are furthermore non-anticipative, are near optimal for the product measure as well. This allows one to approximate a continuous-time process with a (sampled) discrete-time process and the machinery developed for discrete-time optimal control will be applicable. 
\qed

Exploiting the near optimality of the piece-wise constant policies, in the next section we apply discrete-time approximations.

\section{Near Optimality of Control Policies Designed for Discrete-time Models via Sampling}

In view of the results presented, we have that piece-wise constant policies are near optimal. These then lead to discrete-time models whose solutions will be near optimal, and applicable to the original problems. 

For each of the information structures below, we will consider the following arguments: (i) We obtain a sequence of discrete-time models ${\cal T}_n$ arrived from piece-wise constant control policies applied to the true model ${\cal T}$. (ii) We show that the solution to the optimal discrete-time model $\cJ_{T}({\cal T}_n)$ leads to a solution which is near optimal. The direction, $\lim_{n \to \infty} \cJ_{T}({\cal T}_n) \leq \cJ_{T}(\cal  T) + \epsilon$ follows from the analysis above and the direction $\lim_{n \to \infty} \cJ_{T}({\cal T}_n) \geq \cJ_{T}(\cal  T)$ follows from the fact that restricting to piece-wise constant policies cannot lead to a better policy when compared with arbitrary admissible policies. (iii) We then show that the policy obtained to solve $\cJ_{T}({\cal T}_n)$ can be applied to ${\cal T}$ (and thus the approach is constructive) and is near optimal for large $n$.

\subsection{Fully Observed Setup}
Consider the fully observed setup discussed in Section \ref{sectionFullyObserved} with model (\ref{diffProcess}) and cost criterion (\ref{cost1}). That is, with dynamics \[dX_t = b(X_t,U_t) + \upsigma(X_t)dW_t,\]
and cost criterion
\[\cJ_{T}(x, U) = \Exp_x^{U}[\int_0^T c(X_s,U_s) ds + c_{T}(X_T)].\]

{\bf Discrete-Time Model to be Solved for Near Optimal Solutions:}
Let $X^h, U^h$ be the solution of the sampled process corresponding to (\ref{diffProcess}) with piece-wise constant policies.

We define the transition probabilities as follows: For any $k\in\ZZ_{+}$,  distribution of the state $X_k^h$ conditioned on the past state and action variables, is determined by the diffusion process (\ref{diffProcess}), such that, conditioned on $(X_{k-1}^h,\dots,X_0^h,U^h_{k-1},\dots,U_0^h)$, $X_k^h$ has the same distribution as 
\begin{align}\label{DTmodelCTContFully}
X_{kh} = X_{k-1}^h +\int_{(k-1)h}^{kh}b(X_s,U^h(s))ds+\int_{(k-1)h}^{kh}\upsigma(X_s)dW_s
\end{align} 
where $U^h(s)=U_{k-1}^h$ for all $(k-1)h \leq s < kh$ (that is the control $U^h$ is piece-wise constant in time). Hence, for any $A\in\cB(\Rd)$
\begin{align*}
Pr(X_k^h\in A|X^h_{[0,k-1]},U^h_{[0,k-1]})=\mathcal{T}_h(A|X^h_{k-1},U^h_{k-1})
\end{align*}
where ${X^h}_{[0,k-1]},{U^h}_{[0,k-1]}:=X_0^h,\dots,X^h_{k-1},U^h_0,\dots,U^h_{k-1}$, such that 
\begin{align*}
X_k^h \sim \mathcal{T}_h(dx_k|X^h_{k-1},U^h_{k-1})
\end{align*}
where $X_k^h$ determined by (\ref{DTmodelCTContFully}) and where  $\mathcal{T}_h$ is the transition kernel of the Markov chain which is a stochastic kernel from $\Rd\times
\Act$ to $\Rd$.

For the discrete-time model, an {\em admissible policy} is a sequence of control functions $\{U^h_k,\, k\in \ZZ_{+}\}$ such that $U^h_k$ is measurable with respect to the $\sigma$-algebra generated by the information variables
$
I_k^h=\{X_{[0,k]}^h,U_{[0,k-1]}^h\},\,\, k \in \NN,\,\, I_0^h=\{X_0^h\},
$ that is
\begin{equation}
\label{eq_control}
U_k^h=v_k^h(I_k^h),\quad k\in \ZZ_{+},\nonumber
\end{equation}
where $v_k^h$ is a $\mathcal{P}(\Act)$-valued measurable function for $k\in \ZZ_{+}$\,. We define $\Uadm^h$ to be the set of all such admissible policies. Let $\Usm^h := \{U^h\in \Uadm^h: U_k^h = v^h(X_k^h)\,\,\text{for some measurable map}\,\, v^h:\Rd\to \mathcal{P}(\Act)\}$ be the set of all stationary Markov policies\,.

The cost can be written as, with $N_h = \frac{T}{h}$,
\begin{equation}\label{EDisCost1A}
\cJ_{T,h}(U^h) := \Exp_{x}^{U^h}[\sum_{k=0}^{N_h - 1} \hat{c}(X_{kh},U_{kh}) + c_T(X_{N_hh})] 
\end{equation}
where
\[\hat{c}(x,\zeta) =\Exp[\int_0^h c(X_s,U_{0}) ds | X_{0}=x,U_{0}=\zeta]\]
With $n=\frac{1}{h}$, the above define a discrete-time MDP with transition kernel $\mathcal{T}_n$, cost function $\tilde{c}_n$ and total cost $\cJ_{T,n}(x,U)$.

The information structure at time $t$ contains the continuous-time measurements. However, since for such fully observed model, Markov policies are optimal, it suffices to the controller to only use the discrete-time measurements. 

\begin{theorem}\label{TNearOptimFull}
Suppose that Assumptions \hyperlink{A1}{{(A1)}} and \hyperlink{A2}{{(A2)}} hold. Then
\begin{itemize}
\item[(i)] The value of the discrete-time model convergences to the value of the original continuous-time model. 
\item[(ii)] For every $\epsilon > 0$, there exists $h > 0$ so that the solution of the discrete-time approximation gives a policy which is near optimal for the original continuous-time model.
\end{itemize}
\end{theorem}

\begin{proof}
From Theorem~\ref{existencewithMeasTransformation}, we have for every $\epsilon >0$ there exists a piece-wise constant non-anticipative policy $m^\epsilon\in \Uwrc$ which is $\epsilon$-optimal, i.e., $$\cJ_T(x,m^{\epsilon}) \leq  \inf_{m\in \Uwrc} \cJ_T(x,m) + \epsilon\,.$$ Now using this piece-wise constant policy and following the discretization approach as in (\ref{DTmodelCTContFully}) we obtain a discrete-time model $(X^n, {\cal T}_n)$ with the associated total cost $\cJ_{T,n}(x,m^\epsilon)$ given by (\ref{EDisCost1A})\,. Let $\cJ_{T,n}^*$ denotes the optimal (minimum over all possible admissible policies) cost of the discrete-time model. Since, $\cJ_{T,n}^*(x) \leq \cJ_{T,n}(x, m^\epsilon)$, thus it follows that $\lim_{n\to\infty} \cJ_{T,n}^*(x) \leq \inf_{m\in \Uwrc} \cJ_T(x, m) + \epsilon$\,. In view of the fact that restricting to piece-wise constant policies cannot lead to a better policy when compared with arbitrary admissible policies, we deduce that $\lim_{n\to\infty} \cJ_{T,n}^*(x) \geq \inf_{m\in \Uwrc} \cJ_T(x, m)$\,. Since, $\epsilon$ is arbitrary, we obtain $\lim_{n\to\infty}\cJ_{T,n}^*(x) = \inf_{m\in \Uwrc} \cJ_T(x, m)$. Let $v^{n*}$ be an optimal control for the discretized model and $v^{n*}(\cdot)$ be the associated continuous-time interpolated control. Now, by triangle inequality
\begin{equation}\label{ETNearOptimFull1A}
 |\inf_{m\in \Uwrc} \cJ_T(x,m) - \cJ_T(x, v^{n*}(\cdot)) | \leq |\inf_{m\in \Uwrc} \cJ_T(U) - \cJ_{T,n}^*(x) | + |\cJ_{T,n}^*(x) - \cJ_T(x,v^{n*}(\cdot)) |\,.    
\end{equation} We have already shown that the first term of the above inequality goes to $0$. Also, since $\cJ_{T,n}^*(x) = \cJ_{T,n}(x,v^{n*})$ and $\cJ_{T,n}(x, v^{n*})$ is the discretization of $\cJ_T(x, v^{n*}(\cdot))$, we have that $\cJ_{T,n}^*(x)$ and $\cJ_T(x, v^{n*}(\cdot))$ converge to the same limit as $n\to \infty$\,. Thus, we conclude that
$$\lim_{n\to \infty} \cJ_T(x,v^{n*}(\cdot)) = \inf_{m\in \Uwrc} \cJ_T(x,m)\,.$$ This completes the proof of the theorem\,.
\end{proof}

The above apply also to the partially observed and decentralized models.

\subsection{Partially Observed Setup}
Consider the setup in Section \ref{partiallyObsSec} with dynamics (\ref{EPartialStat1}) and criterion \ref{criterionPOMDPCT}.

{\bf Discrete-Time Model to be Solved for Near Optimal Solutions:}
As in the fully observable case, we consider be a piece-wise constant control policy  $U^{h}$ with
\[U^{h}(s) = U_{kh}, \qquad kh \leq s < (k+1)h\]
and let $X^h,Y^h$ be the solution of the sampled process corresponding to (\ref{EPartialStat1}) with piece-wise constant policy $U^{h}$. We have that
\begin{align}
X_{(k+1)h} &= X_{kh} + \int_{kh}^{(k+1)h} b(X_{s},U_{kh})ds + \int_{kh}^{(k+1)h} \upsigma(X_{s}) dW_s \nonumber \\
Y_{t} &= Y_{kh} + \int_{kh}^{t} g(X_s) ds + \int_{kh}^{t} dB_s, \quad t \in [kh, (k+1)h) \label{sdePartial2}
\end{align}

The information structure at time $t$ contains the continuous-time measurements. However, unlike the fully observed model considered above, for partially observed models we cannot only consider the discrete-time measurements and accordingly the measurements are to be continuous-time measurements.

Thus, one defines a discrete-time model in which $X_k^h:=X_{kh}$ and $U_k^h:= U_{kh}$ for $k \in \mathbb{Z}_+$, and the path-valued discrete-time measurement ${\bar Y}_k^h = Y_{[kh,(k+1)h)}$, for $k \in \mathbb{Z}_+$. The cost can be written as, with $N_h = \frac{T}{h}$,
\begin{align}\label{discretetimeApprCost}
\Exp_{x}^{U^h}[\sum_{k=0}^{N_h - 1} \hat{c}(X_{kh},U_{kh}) + c_{T}(X_{N_hh})] 
\end{align}
with
\[\hat{c}(x,\zeta) = \Exp_{x}[\int_0^h c(X_s,U_{0}) ds | X_{0}=x,U_{0}=\zeta]\]
With $n=\frac{1}{h}$, the above define a discrete-time POMDP with transition kernel ${\cal T}_n$ and cost function $\tilde{c}_n$. 
Following the proof technique as in Theorem~\ref{TNearOptimFull} (and Theorem \ref{existencewithMeasTransformation}), we obtain the following near- optimality result for partially observable model\,.
\begin{theorem}\label{TNearOptimPartial}
Suppose that Assumptions \hyperlink{A1}{{(A1)}} and \hyperlink{A2}{{(A2)}} hold\,. Then
\begin{itemize}
\item[(i)] The optimal value of the discrete-time model convergences to the optimal value of the original continuous-time model. 
\item[(ii)] For every $\epsilon > 0$, there exists $h > 0$ so that the solution of the discrete-time approximation gives a policy which is near optimal for the original continuous-time model.
\end{itemize}
\end{theorem}
The result above, however, while involves discrete-time control and state, requires having access to the path-valued measurements. It would be desirable to obtain a discrete-time model with discrete-time measurements $Y_{kh}$ in the original measurement space. That is, at time $kh$, we would like to have $U_{kh}=\gamma(Y_{ih}, i \in \{0,1,\cdots, k\})$ under an admissible policy $\gamma$ (which is a $\Prob(\Act)$ valued measurable map)\,. The following result, building on Lusin's theorem, achieves this. The method involves:
\begin{itemize}
\item Establishing a new probability space where measurements are independent from the state process.
\item Demonstrating that any measurable policy remains continuous on a set with measure $(1-\epsilon)$ for any $\epsilon >0$ within the measurement process space.
\item Defining the measurement process at each discrete sample point and then considering the discretized measurements.
\end{itemize}
As in Section \ref{partiallyObsSec}, we apply measure transformation but to the discrete-time model (\ref{sdePartial2}), then we obtain the equivalent discrete-time model below:
\begin{align}
X_{(k+1)h} &= X_{kh} + \int_{kh}^{(k+1)h} b(X_{s},U_{kh})ds + \int_{kh}^{(k+1)h} \upsigma(X_{s}) dW_s \nonumber \\
Y_{t} &= Y_{kh} + \int_{kh}^{t} dB_s, \quad t \in [kh, (k+1)h) \label{sdeApr2}
\end{align}

In the above, we view the measurement in the interval $[kh,(k+1)h]$ as a continuous path-valued measurement. Measure change with discrete-time measurements, inspired from Girsanov, have been adopted extensively, see \cite{wit88} (see \cite{hogeboom2021sequential} for a review). Note that the absolute continuity condition to allow for measure transformation is applicable for the discrete-time setup as a corollary of the continuous-time setting Section \ref{partiallyObsSec}. In this case, the policies 
$$U_{kh} = \gamma(\{Y_s, s \leq kh\})$$
are such that on a set of measure $1-\epsilon,$ $\gamma=\gamma_C^{\epsilon}$ for a continuous function $\gamma_C^{\epsilon}$ (for any $\epsilon > 0$). Furthermore, the process $Y_{[0,kh]}$ and its piece-wise constant interpolation $\tilde{Y}_t^{h} = \sum_{i=1}^{k} Y_{ih} 1_{t \in [ih,(i+1)h]}$ are so that as $h \to 0$
$$ \|Y - \tilde{Y}^{h}\| \to 0.$$
    The absolute continuity condition permitting measure transformation is applicable in the discrete-time setting as a consequence of the continuous-time framework discussed in Section \ref{partiallyObsSec}.

Accordingly, we have that the cost (\ref{discretetimeApprCost}) under $\gamma$ is at most $2 (k+1)\|c\|_{\infty} \epsilon$ different from the cost of its continuous approximation $\gamma_C$ and the continuous policy $\gamma_C$ asymptotically, as $h \to 0$, attains the same cost by having access to the sampled measurements; see \cite[Prop. 6.1]{saldiyukselGeoInfoStructure}. We then state the following refinement to Theorem \ref{TNearOptimPartial} with a truly regular discrete-time Partially Observable Markov Decision Process (POMDP) approximation\,.
    
\begin{theorem}
Suppose that Assumptions \hyperlink{A1}{{(A1)}} and \hyperlink{A2}{{(A2)}} hold\,. Then
\begin{itemize}
\item[(i)] The optimal value of the discrete-time model (\ref{sdeApr2}) convergences to the optimal value of the original continuous-time model. 
\item[(ii)] For every $\epsilon > 0$, there exists $h > 0$ so that the solution of the discrete-time approximation gives a policy (i.e., a policy $U_{kh}^{\epsilon}=\gamma^{\epsilon}(Y_{ih}, i \in \{0,1,\cdots, k\})$ obtained in (\ref{sdeApr2})) which is near optimal for the original continuous-time model.
\end{itemize}
\end{theorem}

\subsection{Decentralized Setup}

\subsubsection{Decentralized Model with Local Measurements:}
Consider Section \ref{CoupledLocalMeas} with dynamics (\ref{sdeDec1}) and cost criterion (\ref{decCostCrit}). In particular the model is
\begin{align}
dX_t &= b(X_t,U^1_t,\cdots,U^N) dt + \upsigma(X_t)dW_t, \nonumber \\
dY^i_t &= g^i(X_t) dt + dB^i_t, \qquad i=1,\cdots, N.\label{sdeDec21}
\end{align}

{\bf Discrete-Time Model to be Solved for Near Optimal Solutions.}
Let us consider a piece-wise constant control policy with
\[U^{i,h}(s) = U^i_{kh}, \qquad kh \leq s < (k+1)h, \qquad \qquad i=1,\cdots, N\,. \]
Let $X^h,Y^{i,h}$ be the associated sampled process corresponding to (\ref{sdeDec21}) under the piece-wise constant policy $U^{i,h}, i=1,\cdots, N\,.$ Which is sampled as follows:
\begin{align}
X_{(k+1)h} &= X_{kh} + \int_{kh}^{(k+1)h} b(X_{s},U^1_{kh},\cdots,U^N_{kh})ds + \int_{kh}^{(k+1)h} \upsigma(X_{s}) dW_s \nonumber \\
Y^i_{t} &= Y^i_{kh} + \int_{kh}^{t} g^i(X_s) ds + \int_{kh}^{t} dB^i_s, \quad t \in [kh, (k+1)h). \label{sde2A}
\end{align}

The information structure at time $t$ contains the continuous-time measurements. Again, unlike the fully observed model considered above, for decentralized or partially observed models we cannot only consider the discrete-time measurements and accordingly the measurements are to be continuous-time measurements. But, by the similar argument as in the partially observable case, applying Girsanov's change of measure argument to the discretized model (\ref{sde2A}), 
we can define a new probability measure space in which the measurements of the each individual are independent of the state process. In particular, we obtain the following equivalent discretized model
\begin{align}\label{EDisNMesA}
X_{(k+1)h} &= X_{kh} + \int_{kh}^{(k+1)h} b(X_{s},U^1_{kh},\cdots,U^N_{kh})ds + \int_{kh}^{(k+1)h} \upsigma(X_{s}) dW_s \nonumber \\
Y^i_{t} &= Y^i_{kh} + \int_{kh}^{t} dB^i_s, \quad t \in [kh, (k+1)h)\,,
\end{align}
with policy $U_{kh}^{i} = \gamma^{i}(Y_s^{i}, s\leq kh)$, by Lusin's theorem, for some continuous function $\gamma_c^i$ we have $\gamma^i = \gamma_c^i$ on a set of measure $(1-\epsilon_i)$\,. Moreover, we have that the process $Y_{[0, kh]}^i$ can be approximated by its piece-wise constant interpolations. Since both running/terminal costs are continuous, the cost (\ref{EDisDIsCost}) under a policy $\gamma = (\gamma^1,\dots , \gamma^N)$ (under continuous-time measurements) and its continuous approximation $\gamma_c = (\gamma_c^1, \dots , \gamma_c^N)$ (with discrete-time measurements) are close to each other\,. This enables us to obtain a discrete-time model with discrete-time measurements\,.

Thus, one defines a discrete-time model in which $X_k^h:=X_{kh}$ and $U^{i,h}_k:= U^i_{kh}$ for $k \in \mathbb{Z}_+$, and the path-valued discrete-time measurement ${\bar Y}^{i,h}_k=Y^i_{[kh,(k+1)h)}$, for $k \in \mathbb{Z}_+$. The cost can be written as, with $N_h = \frac{T}{h}$,
\begin{equation}\label{EDisDIsCost}
\Exp_{x}^{{\bf U}^h}[\sum_{k=0}^{N_h - 1} \hat{c}(X_{kh},U^1_{kh},\cdots,U^N_{kh}) + c_{T}(X_{N_hh})] 
\end{equation}
with
\[\hat{c}(x,{\bf \zeta}) = \Exp_{x}[\int_0^h c(X_s,U^1_{0}, \cdots, U^N_0) ds | X_{0}=x,{\bf U}_{0}={\bf \zeta}]\]
With $n=\frac{1}{h}$, the above define a discrete-time decentralized POMDP. Again, for the decentralized model, similar proof technique as in Theorem~\ref{TNearOptimFull}, gives us the following near- optimality result\,.
\begin{theorem}\label{TNearOptimDecentralized}
Suppose that Assumptions \hyperlink{D1}{{(D1)}}, \hyperlink{D2}{{(D2)}} and \hyperlink{D3}{{(D3)}} hold. Then
\begin{itemize}
\item[(i)] The optimal value of the discrete-time model (\ref{EDisNMesA}) convergences to the optimal value of the original continuous-time model. 
\item[(ii)] For every $\epsilon > 0$, there exists $h > 0$ so that the solution of the discrete-time approximation gives a policy (i.e., a policy $U_{kh}^{i,\epsilon}=\gamma^{i,\epsilon}(Y_{rh}, r \in \{0,1,\cdots, k\})$, $i=1,\cdots ,N$ an optimal solution of (\ref{EDisNMesA})) which is near optimal for the original continuous-time model.
\end{itemize}
\end{theorem}

\subsection{Decentralized Model with Coupled Dynamics and Local State}
We finally consider the model given in Section \ref{CoupledDyLocalSt}.

{\bf Discrete-Time Model to be Solved for Near Optimal Solutions:}
For $1 \leq i \leq N$, let $X^{i,h}$ be the solution of the sampled process corresponding to (\ref{sdeDecLoc1}) with piece-wise constant policies:
\[U^{i,h}_{s} = U^{i}_{kh}, \qquad kh \leq s < (k+1)h,\]
which is given as follows: for $t \in [kh,(k+1)h)$
\begin{align}\label{DTmodelCTConDecCoup1}
X^i_{t} = X^i_{kh} + \int_{kh}^{t} b^i(X^i_{s},U_{kh})ds + \int_{kh}^{t} b^i_0({\bf X}_{s},{\bf U}_{kh})ds+ \int_{kh}^{t} \upsigma(X^i_{s}) dW_s\,.
\end{align}

The information structure at time $t$ contains the continuous-time measurements. Again, unlike the fully observed model considered above, for decentralized or partially observed models we cannot only consider the discrete-time measurements and accordingly the measurements are to be continuous-time measurements. Thus, one defines a discrete-time model in which $X_k^h:=X_{kh}$ and $U^{i,h}_k:= U^i_{kh}$ for $k \in \mathbb{Z}_+$, and the path-valued discrete-time measurement ${\bar X}^{i,h}_k=X^i_{[kh,(k+1)h)}$, for $k \in \mathbb{Z}_+$.
By Girsanov's change of measure argument, under a new probability measure we can have our model 
\begin{eqnarray}\label{DTmodCTContApA}
X^i_{t} = X^i_{kh} + \int_{kh}^{t} \upsigma(X^i_{s}) dW_s, \quad t \in [kh,(k+1)h) 
\end{eqnarray}
which is equivalent to (\ref{DTmodelCTConDecCoup1}). This makes the measurements to be just function of the local states and independent of the actions\,. Then, as in the partially observable case, the Lusin's theorem and piece-wise constant interpolation of the measurements, enable us to consider discrete-time model with discrete-time measurements\,.

The cost can be written as, with $N_h = \frac{T}{h}$,
\[\Exp_{x}^{{\bf U}^h}[\sum_{k=0}^{N_h - 1} \hat{c}({\bf X}_{kh},{\bf U}_{kh}) + c_T({\bf X}_{N_hh})] \]
with
\[\hat{c}({\bf x},{\bf \zeta}) =\Exp_{x}[\int_0^h c({\bf X}_s,{\bf U}_{0}) ds | {\bf X}_{0}={\bf x},{\bf U}_{0}={\bf \zeta}]\]
With $n=\frac{1}{h}$, the above define a discrete-time MDP. For this decentralized model with coupled dynamics and local state, following steps as in the proof of Theorem~\ref{TNearOptimFull}, we obtain the following near- optimality result\,.
\begin{theorem}\label{TNearOptimDecenLocal}
Suppose that Assumptions \hyperlink{D1}{{$\hat{(D1)}$}}, \hyperlink{D2}{{$\hat{(D2)}$}} and \hyperlink{D3}{{$\hat{(D3)}$}} hold. Then
\begin{itemize}
\item[(i)] The optimal value of the discrete-time model (\ref{DTmodCTContApA}) convergences to the optimal value of the original continuous-time model. 
\item[(ii)] For every $\epsilon > 0$, there exists $h > 0$ so that the solution of the discrete-time approximation gives a policy which is near optimal for the original continuous-time model.
\end{itemize}
\end{theorem}


\section{Conclusion}
We presented existence and discrete-time approximation results on optimal control policies for continuous-time stochastic control problems under a variety of information structures. These include fully observed models, partially observed models and decentralized information structures. While there exist results for the fully observed setup, the discrete-time approximation results for partially observed and decentralized models appear to be new. These results lead to the applicability of well-established partially observed Markov decision processes and the relatively more mature theory of discrete-time decentralized stochastic control to be applicable for computing near optimal solutions and also facilitate learning theoretic approaches.


\section{Appendix}
\textbf{Proof of Lemma~\ref{SomnathLemma}:}
Let $m^n \to m$ in $\Urc$\,.\\

{\bf Step 1.}
First note that for any $a, b \in \mathbb{R}$
\begin{equation}\label{ESOMEST}
e^b - e^a = \int_a^b e^{x}dx \leq |b-a| \max(e^b,e^a).
\end{equation}
Thus, 
\begin{align}
& |e^{(\int_0^T \upsigma^{-1}(X_s)b(X_s,m^n_s)dW_s )} - e^{(\int_0^T \upsigma^{-1}(X_s)b(X_s,m_s)dW_s )}|  \nonumber \\
& \leq | \int_0^T \upsigma^{-1}(X_s)b(X_s,m^n_s)dW_s - \int_0^T \upsigma^{-1}(X_s)b(X_s,m_s)dW_s|  \nonumber \\
& \quad \quad \times \max\bigg(e^{(\int_0^T \upsigma^{-1}(X_s)b(X_s,m^n_s)dW_s)},e^{(\int_0^T \upsigma^{-1}(X_s)b(X_s,m_s)dW_s)}\bigg)
\end{align}
Taking the expectation and by the Cauchy-Schwarz inequality:
\begin{align}\label{step1BoundUnif}
& \Exp_{x}\bigg[\bigg|e^{(\int_0^T \upsigma^{-1}(X_s)b(X_s,m^n_s)dW_s )} - e^{(\int_0^T \upsigma^{-1}(X_s)b(X_s,m_s)dW_s )}\bigg|\bigg]^2  \nonumber \\
& \leq \Exp\bigg[| \int_0^T \upsigma^{-1}(X_s)b(X_s,m^n_s)dW_s - \int_0^T \upsigma^{-1}(X_s)b(X_s,m_s)dW_s|^2\bigg]  \nonumber \\
& \quad \quad \Exp\bigg[|\max\bigg(e^{(\int_0^T \upsigma^{-1}(X_s)b(X_s,m^n_s)dW_s)},e^{(\int_0^T \upsigma^{-1}(X_s)b(X_s,m_s)dW_s)}\bigg)|^2 \bigg]
\end{align}

{\bf Step 2.}
The final term on the right in (\ref{step1BoundUnif}) is bounded, uniformly, over control policies. To see this, we will consider
\begin{align}\label{eqnToBeStudied}
e^{(\int_0^T \upsigma^{-1}(X_s)b(X_s,m_s)dW_s)}
\end{align}
as the solution of the following stochastic differential equation. We first define
\begin{align}
 d\hat{X}_s :=  \upsigma^{-1} ( X_s )b(X_s,m_s)dW_s 
 \end{align}
\[Y_s:=e^{\hat{X}_s},\]
which is then the process (\ref{eqnToBeStudied}) we want to study; in particular we wish to show that $\Exp[Y^2]$ is uniformly bounded over control policies. By It\^o,
\begin{align*}
dY = \frac{dY}{d\hat{X}} d\hat{X} + \frac{1}{2} \frac{d^2Y}{d\hat{X}^2} d\hat{X}\cdot d\hat{X} 
\end{align*} Thus we obtain
\begin{align}
dY_s &= e^{\hat{X}_s} (\upsigma^{-1}(X_s)b(X_s,m_s)dW_s) + \frac{1}{2} e^{\hat{X}_s}|\upsigma^{-1}(X_s)b(X_s,m_s)|^2 ds \nonumber \\
&= Y_s  \bigg(\upsigma^{-1}(X_s)b(\hat{X}_s,m_s)dW_s + \frac{1}{2}|\upsigma^{-1}(X_s)b(X_s,m_s)|^2ds \bigg) 
\end{align}

Now, by another It\^o,
\begin{align}
d(Y^2_s) & = 2Y_s dY_s + dY_s\cdot dY_s \nonumber\\ 
& = 2Y^2_s \bigg(\upsigma^{-1}(\hat{X}_s)b(\hat{X}_s,m_s)dW_s + \frac{1}{2}|\upsigma^{-1}(X_s)b(X_s,m_s)|^2ds \bigg) + Y^2_s|\upsigma^{-1}(X_s)b(X_s,m_s)|^2ds \nonumber \\
& = 2Y_s^2 \upsigma^{-1}(\hat{X}_s)b(\hat{X}_s,m_s)dW_s + 3Y_s^2 |\upsigma^{-1}(X_s)b(X_s,m_s)|^2 ds\,.  \nonumber
\end{align}
Thus, $V = Y^2$ solves the stochastic differential equation
\begin{align*}
dV_s &=  3V_s|\upsigma^{-1}(X_s)b(X_s,m_s)|^2ds +   2V_s \upsigma^{-1}(X_s)b(X_s,m_s)dW_s  \,.
\end{align*}
Thus, $V$ is so that it is bounded almost surely for finite intervals and its expectation is also finite by \cite[Theorem 2.3.1]{kushner2012weak} (also see \cite[Theorem~2.2.2]{ABG-book}). Thus, 
\[\Exp\bigg[|\max\bigg(e^{(\int_0^T \upsigma^{-1}(X^n_s)b(X^n_s,m^n_s)dW_s)},e^{(\int_0^T \upsigma^{-1}(X_s)b(X_s,m_s)dW_s)}\bigg)|^2 \bigg]\]
is uniformly bounded over all control policies.

%
%
%
%
%

{\bf Step 3}.
By (\ref{step1BoundUnif}), it then suffices to show that the first term $\Exp\bigg[ \bigg| \int_0^T \upsigma^{-1}(X_s)b(X_s,m_s)dW_s \bigg|^2\bigg]$ is continuous in the policy. This follows from the It\^o isometry and the proof steps presented in the proof of Theorem \ref{contConvPathM}. Since $\upsigma^{-1}b$ is bounded and continuous, in view of (\ref{step1BoundUnif}), 
$e^{- \frac{1}{2} \int_0^T |\upsigma^{-1}(X_s)b(X_s,m^n_s)|^2 ds}$ converges to $e^{- \frac{1}{2} \int_0^T |\upsigma^{-1}(X_s)b(X_s,m_s)|^2 ds}$ in $L^1$ as $n\to \infty$\,.\qed

\bibliographystyle{plain}
\bibliography{Quantization,SerdarBibliography}

\begin{thebibliography}{10}

\bibitem{ABG-book}
A.~Arapostathis, V.~S. Borkar, and M.~K. Ghosh.
\newblock {\em Ergodic control of diffusion processes}, volume 143 of {\em
  Encyclopedia of Mathematics and its Applications}.
\newblock Cambridge University Press, Cambridge, 2012.

\bibitem{balder1997consequences}
E.~J. Balder.
\newblock Consequences of denseness of dirac young measures.
\newblock {\em Journal of Mathematical Analysis and Applications},
  207(2):536--540, 1997.

\bibitem{BJ-06}
G.~Barles and E.~R. Jakobsen.
\newblock Error bounds for monotone approximation schemes for
  hamilton-jacobi-bellman equations.
\newblock {\em SIAM Journal on Numerical Analysis}, 43(2):540--558, 2006.

\bibitem{beiglbock2018denseness}
M.~Beiglb{\"o}ck and D.~Lacker.
\newblock Denseness of adapted processes among causal couplings.
\newblock {\em arXiv}, pages arXiv--1805, 2018.

\bibitem{benevs1971existence}
V.~E. Bene{\v{s}}.
\newblock Existence of optimal stochastic control laws.
\newblock {\em SIAM Journal on Control}, 9(3):446--472, 1971.

\bibitem{PBill-book}
P.~Billingsley.
\newblock {\em Convergence of Probability Measures}.
\newblock 2nd ed. Wiley, New York, 1999.

\bibitem{bismut1982partially}
J.-M. Bismut.
\newblock Partially observed diffusions and their control.
\newblock {\em SIAM Journal on Control and Optimization}, 20(2):302--309, 1982.

\bibitem{borkar1988probabilistic}
V.~S. Borkar.
\newblock The probabilistic structure of controlled diffusion processes.
\newblock {\em Acta Applicandae Mathematica}, 11(1):19--48, 1988.

\bibitem{borkar1989topology}
V.~S. Borkar.
\newblock A topology for {M}arkov controls.
\newblock {\em Applied Mathematics and Optimization}, 20(1):55--62, 1989.

\bibitem{castaing2004young}
C.~Castaing, P.~R.~De Fitte, and M.~Valadier.
\newblock {\em Young measures on topological spaces: with applications in
  control theory and probability theory}, volume 571.
\newblock Springer Science \& Business Media, Dordrecht, 2004.

\bibitem{CharalambousA13}
C.D.Charalambous and N.U. Ahmed.
\newblock Dynamic team theory of stochastic differential decision systems with
  decentralized noisy information structures via girsanov's measure
  transformation.
\newblock {\em arXiv}, abs/1309.1913, 2013.

\bibitem{charalambous2016decentralized}
C.~D. Charalambous.
\newblock Decentralized optimality conditions of stochastic differential
  decision problems via {G}irsanov's measure transformation.
\newblock {\em Mathematics of Control, Signals, and Systems}, 28(3):1--55,
  2016.

\bibitem{charalambous2014equivalence}
C.~D. Charalambous and N.~U. Ahmed.
\newblock Equivalence of decentralized stochastic dynamic decision systems via
  {G}irsanov's measure transformation.
\newblock In {\em IEEE Conference on Decision and Control (CDC)}, pages
  439--444. IEEE, 2014.

\bibitem{charalambous2014maximum}
C.~D. Charalambous and N.~U. Ahmed.
\newblock Maximum principle for decentralized stochastic differential decision
  systems.
\newblock In {\em IEEE Conference on Decision and Control (CDC)}, pages
  1846--1851. IEEE, 2014.

\bibitem{Crisan2022}
D.~Crisan, A.~Lobbe, and S.~Ortiz-Latorre.
\newblock {\em Pathwise Approximations for the Solution of the Non-Linear
  Filtering Problem}.
\newblock Springer International Publishing, Cham, 2022.

\bibitem{MasiRung81}
G.~B. Di~Masi and W.~J. Runggaldier.
\newblock Continuous-time approximations for the nonlinear filtering problem.
\newblock {\em Appl. Math. Optim.}, 7(3):233--245, 1981.

\bibitem{di1982approximation}
G.B. Di~Masi and W.J. Runggaldier.
\newblock On approximation methods for nonlinear filtering.
\newblock In {\em Nonlinear Filtering and Stochastic Control: Proceedings of
  the 3 rd 1981 Session of the Centro Internazionale Matematico Estivo (CIME),
  Held at Cortona, July 1--10, 1981}, pages 249--259. Springer, 1982.

\bibitem{FlPa82}
W.H. Fleming and E.~Pardoux.
\newblock Optimal control for partially observed diffusions.
\newblock {\em SIAM J. Control Optim.}, 20(2):261--285, 1982.

\bibitem{gupta2014existence}
A.~Gupta, S.~Y\"uksel, T.~Ba\c{s}ar, and C.~Langbort.
\newblock On the existence of optimal policies for a class of static and
  sequential dynamic teams.
\newblock {\em SIAM Journal on Control and Optimization}, 53:1681--1712, 2015.

\bibitem{HoChu}
Y.~C. Ho and K.~C. Chu.
\newblock Team decision theory and information structures in optimal control
  problems - part {I}.
\newblock {\em IEEE Transactions on Automatic Control}, 17:15--22, February
  1972.

\bibitem{hogeboom2021sequential}
I.~Hogeboom-Burr and S.~Y\"uksel.
\newblock Sequential stochastic control (single or multi-agent) problems nearly
  admit change of measures with independent measurements.
\newblock {\em Applied Mathematics and Optimization}, 2023.

\bibitem{huang2022general}
J.~Huang, G.~Wang, and W.~Wang.
\newblock A general linear quadratic stochastic control and information value.
\newblock {\em Journal of Mathematical Analysis and Applications}, page 126486,
  2022.

\bibitem{IkedaWatanabe89}
N.~Ikeda and S.~Watanabe.
\newblock {\em Stochastic differential equations and diffusion processes},
  volume~24 of {\em North-Holland Mathematical Library}.
\newblock North-Holland Publishing Co., Amsterdam; Kodansha, Ltd., Tokyo,
  second edition, 1989.

\bibitem{jackson2023approximately}
J.~Jackson and D.~Lacker.
\newblock Approximately optimal distributed stochastic controls beyond the mean
  field setting.
\newblock {\em arXiv preprint arXiv:2301.02901}, 2023.

\bibitem{JPR-19P}
E.~R. Jakobsen, A.~Picarelli, and C.~Reisinger.
\newblock Improved order 1/4 convergence for piecewise constant policy
  approximation of stochastic control problems.
\newblock {\em Electronic Communications in Probability}, 24:1--10, 2019.

\bibitem{kara2021convergence}
A.D Kara and S.~Y\"uksel.
\newblock Convergence of finite memory {Q}-learning for {POMDP}s and near
  optimality of learned policies under filter stability.
\newblock {\em Mathematics of Operations Research (also arXiv:2103.12158)},
  2022.

\bibitem{KN98A}
N.~Krylov.
\newblock On the rate of convergence of finite-difference approximations for
  bellman’s equations.
\newblock {\em St. Petersburg Math. J.}, 9:639--650, 1998.

\bibitem{KN99A}
N.~Krylov.
\newblock Approximating value functions for controlled degenerate diffusion
  processes by using piece-wise constant policies.
\newblock {\em Electronic Journal of Probability}, 4:1--19, 1999.

\bibitem{KN2000A}
N.~Krylov.
\newblock On the rate of convergence of finite-difference approximations for
  bellmans equations with variable coefficients.
\newblock {\em Probab Theory Relat Fields}, 117:1--16, 2000.

\bibitem{Krylov}
N.~V. Krylov.
\newblock {\em Controlled diffusion processes}, volume~14 of {\em Applications
  of Mathematics}.
\newblock Springer-Verlag, New York-Berlin, 1980.

\bibitem{kushner1990numerical}
H.~J. Kushner.
\newblock Numerical methods for stochastic control problems in continuous time.
\newblock {\em SIAM Journal on Control and Optimization}, 28(5):999--1048,
  1990.

\bibitem{kushner2001numerical}
H.~J. Kushner and P.~G. Dupuis.
\newblock {\em Numerical Methods for Stochastic Control Problems in Continuous
  Time}, volume~24.
\newblock Springer Science \& Business Media, NY, 2001.

\bibitem{KHJFil79}
Harold~J. Kushner.
\newblock A robust discrete state approximation to the optimal nonlinear filter
  for a diffusion.
\newblock {\em Stochastics}, 3(2):75--83, 1979.

\bibitem{KushnerSingular1990}
H.J. Kushner.
\newblock {\em Weak convergence methods and singularly perturbed stochastic
  control and filtering problems}.
\newblock Springer Science \& Business Media, Birkhäuser Boston, 1990.

\bibitem{kushner2012weak}
H.J. Kushner.
\newblock {\em Weak convergence methods and singularly perturbed stochastic
  control and filtering problems}.
\newblock Springer Science \& Business Media, Birkhäuser Boston, MA, 2012.

\bibitem{lacker2018probabilistic}
D.~Lacker.
\newblock Probabilistic compactification methods for stochastic optimal control
  and mean field games.
\newblock 2018.

\bibitem{MRTSAP09}
M.~Martinez, S.~Rubenthaler, and E.~Tanr\'{e}.
\newblock Approximations of a continuous time filter. {A}pplication to optimal
  allocation problems in finance.
\newblock {\em Stoch. Anal. Appl.}, 27(2):270--296, 2009.

\bibitem{milgrom1985distributional}
P.~R. Milgrom and R.~J. Weber.
\newblock Distributional strategies for games with incomplete information.
\newblock {\em Mathematics of operations research}, 10(4):619--632, 1985.

\bibitem{saldiyukselGeoInfoStructure}
N.~Saldi and S.~Y\"uksel.
\newblock Geometry of information structures, strategic measures and associated
  control topologies.
\newblock {\em Probability Surveys}, 19:450--532, 2022.

\bibitem{saldiyuksellinder2017finiteTeam}
N.~Saldi, S.~Y\"uksel, and T.~Linder.
\newblock Finite model approximations and asymptotic optimality of quantized
  policies in decentralized stochastic control.
\newblock {\em IEEE Transactions on Automatic Control}, 62:2360 -- 2373, 2017.

\bibitem{SelkYuksel2023}
Z.~Selk and S.~Y{\"u}ksel.
\newblock Robust decentralized stochastic control of interacting particles via
  risk sensitive control under measure change.
\newblock 2023.

\bibitem{serfozo1982convergence}
R.~Serfozo.
\newblock Convergence of {L}ebesgue integrals with varying measures.
\newblock {\em Sankhy{\=a}: The Indian Journal of Statistics, Series A}, pages
  380--402, 1982.

\bibitem{warga2014}
J.~Warga.
\newblock {\em Optimal Control of Differential and Functional Equations}.
\newblock Academic press, 2014.

\bibitem{wit88}
H.~S. Witsenhausen.
\newblock Equivalent stochastic control problems.
\newblock {\em Math. Control, Signals and Systems}, 1:3--11, 1988.

\bibitem{YukselWitsenStandardArXiv}
S.~Y\"uksel.
\newblock A universal dynamic program and refined existence results for
  decentralized stochastic control.
\newblock {\em SIAM Journal on Control and Optimization}, 58(5):2711--2739,
  2020.

\bibitem{yuksel2023borkar}
S.~Y{\"u}ksel.
\newblock On {B}orkar and {Y}oung relaxed control topologies and continuous
  dependence of invariant measures on control policy.
\newblock {\em SIAM Journal on Control and Optimization}, 62:2367--2386, 2024.

\bibitem{YukselBasarBook}
S.~Y\"uksel and T.~Ba\c{s}ar.
\newblock {\em Stochastic Networked Control Systems: Stabilization and
  Optimization under Information Constraints}.
\newblock Springer, New York, 2013.

\bibitem{YukselSaldiSICON17}
S.~Y\"uksel and N.~Saldi.
\newblock Convex analysis in decentralized stochastic control, strategic
  measures and optimal solutions.
\newblock {\em SIAM J. on Control and Optimization}, 55:1--28, 2017.

\end{thebibliography}

\end{document}